\documentclass[11pt, reqno]{amsart}
\usepackage{hyperref}
\usepackage[ansinew]{inputenc}	
\usepackage{graphicx}
\usepackage{color}%
\usepackage[numbers, square]{natbib}
\usepackage{mathrsfs}
\usepackage{mathptmx}
\usepackage{amsmath,amsthm,amsfonts,amssymb}
\usepackage{dsfont,extarrows,enumerate}
\usepackage{fullpage}
\usepackage{xcolor}
\usepackage[normalem]{ulem}

\allowdisplaybreaks[4]

\numberwithin{equation}{section}

\makeatletter
\def\@tocline#1#2#3#4#5#6#7{\relax
  \ifnum #1>\c@tocdepth 
  \else
    \par \addpenalty\@secpenalty\addvspace{#2}%
    \begingroup \hyphenpenalty\@M
    \@ifempty{#4}{%
      \@tempdima\csname r@tocindent\number#1\endcsname\relax
    }{%
      \@tempdima#4\relax
    }%
    \parindent\z@ \leftskip#3\relax \advance\leftskip\@tempdima\relax
    \rightskip\@pnumwidth plus4em \parfillskip-\@pnumwidth
    #5\leavevmode\hskip-\@tempdima
      \ifcase #1
       \or\or \hskip 1em \or \hskip 2em \else \hskip 3em \fi%
      #6\nobreak\relax
    \dotfill\hbox to\@pnumwidth{\@tocpagenum{#7}}\par
    \nobreak
    \endgroup
  \fi}
\makeatother

\theoremstyle{plain} \newtheorem{theorem}{Theorem}[section]
\theoremstyle{plain} \newtheorem{proposition}{Proposition}
\theoremstyle{plain} 
\theoremstyle{plain} \newtheorem{lemma}[theorem]{Lemma}
\theoremstyle{plain} \newtheorem{corollary}[theorem]{Corollary}
\theoremstyle{definition} \newtheorem{definition}{Definition}
\theoremstyle{definition} 
\theoremstyle{remark} \newtheorem{remark}[theorem]{Remark}
\theoremstyle{remark} \newtheorem{example}[theorem]{Example}

\newcommand{\R}{\mathbb R}

\newcommand{\NN}{\mathbb N}

\newcommand{\fC}{\mathfrak{C}}

\newcommand{\fK}{\mathfrak{K}}

\newcommand{\sA}{\mathcal{A}}
\newcommand{\sB}{\mathcal{B}}
\newcommand{\sC}{\mathcal{C}}

\newcommand{\sK}{\mathcal K}
\newcommand{\sL}{\mathcal L}
\newcommand{\sM}{\mathcal{M}}
\newcommand{\sP}{\mathcal{P}}

\newcommand{\dodn}{\overset{{\rm d}}\longrightarrow}

\newcommand{\Int}{\mathrm{Int}}

\newcommand{\ext}{\text{ext}}

\DeclareMathOperator{\cch}{\overline{\mathrm{conv}}\,}
\DeclareMathOperator{\ch}{\mathrm{conv}}

\newcommand{\wrap}{{\rm wr}}

\renewcommand{\epsilon}{\varepsilon}

\renewcommand{\emptyset}{\varnothing}
\usepackage{bbm}
\usepackage{fancybox}
\newlength{\querylen}
\title{Peelings and Wrappings of Families of Convex Sets with Applications to Strongly Convex Sets Generated by Random Samples}
\author{Alexander Marynych}
\address{Department of Mathematics, Kyiv School of Economics, Ukraine}
\email{omarynych@kse.org.ua}
\author{Mykyta Sadok}
\address{Faculty of Computer Science and Cybernetics, Taras Shevchenko National University of Kyiv, Kyiv, Ukraine}
\email{sadokmykyta@knu.ua}
\date{\today}

\begin{document}

\begin{abstract}
We introduce and study peeling and wrapping operations for families of compact convex sets. The two peeling procedures considered in the paper are the $m$-point peeling, obtained by intersecting the convex hulls remaining after all possible deletions of $m$ members of the family, and the recursive convex hull peeling, obtained by repeatedly removing the contributing sets, that is, those members whose deletion strictly changes the convex hull. Using polarity, we also introduce the dual wrapping operations for intersections of convex sets.

The deterministic part of the paper develops the geometric framework needed for these constructions. In particular, we study contributing sets under general position assumptions, explain the role of compactness of convex hulls of subfamilies, and prove continuity results for both peeling procedures with respect to a suitable vague convergence of locally finite point measures on the space of compact convex sets.

The probabilistic part applies this framework to $K$-hulls generated by random samples from a convex body $K$. Assuming that $K$ is strictly convex and regular, we prove that the m-point and recursive peelings of the polar bodies associated with the random $K$-hulls converge in distribution to the corresponding peelings of the limiting Poisson object. By polarity, this also yields distributional convergence of the associated wrapping operations for the rescaled random sets themselves.
\end{abstract}



\maketitle

\section{Introduction}\label{introduction}

Classical convex hull peeling starts with fixing a finite point set in $\R^d$, taking its convex hull, removing the extreme points, and then repeating the process. The nested layers produced in this way provide one of the earliest center-outward orderings in multivariate analysis, and they remain a useful geometric model for depth, trimming, and robustness~\cite{barnett1976ordering,chazelle1985convexlayers,eddy1982convex}. Recent research on random peeling has followed different directions: one line of work, for instance, studies the global behavior of many peeling layers and shows that, after rescaling, convex hull peeling admits a continuum limit described by a PDE whose level sets evolve by a power of Gaussian curvature \cite{calder2020limitshape}. Another direction studies the asymptotic behavior of the first peeling layers for Poisson samples, showing that in both smooth and polyhedral convex bodies one can obtain precise expectation and variance asymptotics, together with central limit theorems, for natural functionals of those outer layers \cite{calka2023unitball,calka-simple-poly}. Still, the theory is fundamentally built around convex hulls of point sets. Its main objects are the vertices of ordinary convex hulls, and its recursive structure relies on the combinatorics of finite point configurations.

At the same time, there is a body of literature on strong convexity (aka $K$-convexity), where closed half-spaces in a closed convex hull definition are replaced by translates of a fixed convex body $K$. This leads to $K$-hulls and, more generally, to $K$-strongly convex sets. Their geometry has been studied from several directions, including analytic, combinatorial, and metric ones \cite{balashov2000mstrongly,jahn2017ballconvex,langi2013ballspindle,polovinkin1996strongly}. The studied objects include random disc-polygons, random ball-polytopes, and $K$-hulls generated by random samples \cite{fodor2019ballpolytopes,fodor2014randomdiscpolygons,fodor2020generalised,mm22}. In particular, Marynych and Molchanov developed a facial theory for $K$-strongly convex sets generated by random samples and showed that the model admits a Poisson-type scaling limit~\cite{mm22}. The theory relies heavily on the study of convex hulls of the families of convex sets rather than point configurations. However, once one moves from points to the families of convex sets that naturally appear in these constructions, the classical peeling procedure no longer applies directly. There are several points at which the passage from points to sets creates new difficulties. First, for a general family of convex sets, it is not enough to speak about extreme points of the resulting convex hull. Instead, one has to identify which members of the family actually contribute to the hull. We call a set $L\in\sL$ contributing to a family of convex sets $\sL$ if deleting it strictly shrinks the convex hull of $\sL$. Second, for infinite families, distinctions between algebraic and closed convex hulls, compactness of convex hulls of subfamilies, and general position assumptions become essential. Without such assumptions, recursive peeling may exhibit degenerate behaviour: a layer may fail to be controlled by finitely many sets, or non-contributing sets may nevertheless touch the boundary of the hull. 

We introduce two peeling procedures. The first one is the $m$-point peeling. For a family of convex sets $\sL=\{L_i:i\in I\}$ and $0\leq m<|I| \leq \infty$, it is defined as
$$
\ch^{[m]}(\sL)
:=
\bigcap_{\substack{J\subset I\\ |J|=m}}
\cch\left(\bigcup_{i\in I\setminus J}L_i\right).
$$
Thus $\ch^{[m]}(\sL)$ consists of all points which remain in the closed convex hull no matter which $m$ members of the family are removed. This construction may be viewed as a deletion-robust hull for a family of convex sets. The $m$-point peeling is the part of the closed convex hull that survives the removal of any $m$ bodies. In this sense, it is analogous in spirit to classical robust depth regions and $k$-hull constructions~\cite{cole1987khulls,eddy1982convex}.

The second construction is the recursive convex hull peeling. Starting from $\sL^{(1)}=\sL$, we remove all contributing members of the current family, then repeat the same operation for the remaining family. The corresponding recursive hulls are denoted by
$$
\ch_{[m]}(\sL), \qquad m\geq 1.
$$
When $\sL$ is a finite family of singletons, this procedure reduces to the usual convex hull peeling of a finite point configuration. In the set-valued setting, however, the notion of a contributing set replaces that of an extreme point, and additional assumptions are needed to ensure that the recursion behaves regularly. 

The dual counterpart of these constructions is obtained by polarity. Since polarity transforms intersections into closed convex hulls of unions and reverses inclusions, peelings of the polar family correspond to enlargements of the original intersection. We call the resulting dual operations wrappings. Thus, whereas peeling removes outer layers from a convex hull, wrapping enlarges an intersection by allowing deletions of constraints. We define both $m$-point wrappings and recursive wrappings, and later use them to translate peeling convergence for polar bodies back into convergence statements for intersections.

The deterministic part of the paper establishes the geometric and topological foundation for these operations. We first recall the notion of general position for families of compact convex sets and develop several basic facts about contributing sets. In particular, we show that, under compactness and general position assumptions, non-contributing members lie in the relative interior of the hull generated by the remaining family, and the whole convex hull is already generated by the contributing members. We then introduce a cofinite version of the general position assumption, suitable for infinite families and for the repeated deletion operations appearing in recursive peeling.

The main deterministic results are continuity theorems for the two peeling procedures. We work with point measures on $\sK^d_0\setminus\{\{0\}\}$, where $\sK^d_0$ denotes the space of compact convex sets containing the origin. The topology is a vague topology adapted to the fact that atoms may accumulate only near the deleted singleton $\{0\}$. Under this convergence, and under appropriate general position and interior assumptions, we prove that the recursive peeling layers eventually stabilize at the level of contributing indices and that the corresponding convex hulls converge in the Hausdorff metric. We also prove the analogous Hausdorff convergence for the $m$-point peeling. This setting is particularly useful for our probabilistic application concerning $K$-hulls of random samples for which we now outline the setup. Let $K\in\sK^d_{(0)}$ be a convex body containing the origin in its interior, and let $\Xi_n=\{\xi_1,\ldots,\xi_n\}$ be a sample of independent points uniformly distributed in $K$. The set $
X_n:=\bigcap_{i=1}^n (K-\xi_i)$ describes the difference between $K$ and the $K$-hull of the sample. It is known that $nX_n$ converges in distribution to the zero cell $Z$ of a Poisson hyperplane tessellation determined by the surface area measure of $K$, see~\cite{mm22}. Equivalently, after polarity, the sets $n^{-1}X_n^{o}=n^{-1}\ch\left(\bigcup_{i=1}^{n}(K-\xi_i)^{o}\right)$ converge to $Z^{o}$, which can be represented as the convex hull generated by a limiting Poisson family of line segments. Assuming that $K$ is strictly convex and regular, we apply the deterministic peeling continuity results to this Poisson limit. We prove that, for every fixed $m$, the first $m$-point peelings of $n^{-1}X_n^{o}$ converge jointly in distribution to the corresponding peelings of $Z^{o}$. We also prove the analogous joint convergence for the first $m$ recursive peeling layers. Finally, applying polarity once more, we obtain the corresponding convergence of the $m$-point and recursive wrappings of the rescaled sets $nX_n$.

The paper is organized as follows. Section~\ref{prelims} contains the deterministic geometric framework. We recall the notion of general position, introduce contributing sets, define the two peeling procedures, and then define their polar duals, the wrapping operations. Section~\ref{main-results} applies these constructions to $K$-convexity and states the main probabilistic limit theorems. Section~\ref{sec:proof} is devoted to the proofs. We first establish the deterministic continuity results for peelings under vague convergence of point measures, and then derive the probabilistic convergence theorems for random $K$-hulls by the continuous mapping argument. Some technical proofs are postponed to the Appendix.

Throughout the paper we use the following notation: $\sK^d$ is the family of all compact convex sets in $\R^d$; $\sK^d_0$ (respectively, $\sK^d_{(0)}$) is the family of all compact convex sets in $\R^d$ containing the origin (respectively, the origin in the interior). The convergence with respect to the Hausdorff metric is denoted by
$\overset{H}{\to}$.

\section{Convex hulls of families of convex sets: general position, contributing sets and peelings}
\label{prelims}
We shall first recall some notions from convex geometry along with the notion of general position of a family of sets from $\sK^d$, first introduced in~\cite{mm22}. In what follows $\sL=\{L_i:i\in I\}$ is a collection of pairwise distinct compact convex sets in $\R^d$. Also $\ch(\sL):=\ch\left(\cup_{i\in I}L_i\right)$ and $\cch(\sL):=\cch\left(\cup_{i\in I}L_i\right)$.

Throughout the paper we work with the families $\sL$ such that
\begin{equation}\label{eq:union_is_compact}
\ch(\sL)\text{ is a compact set and }{\rm dim}\,\ch(\sL)>0,
\end{equation}
where $\dim$ denotes the affine dimension.

Recall that, for a convex compact set $L \in \sK^d$, a set $E \subseteq L$ is called an \emph{exposed face} if there exists a supporting hyperplane $H$ of $L$ such that $E = H \cap L$. Let $A$ be an arbitrary closed convex subset of some exposed face $F$ of $\ch(\sL)$. Put
\begin{equation*}
    \sM(\sL,\, A) := \{ L \in \sL: L \cap A \neq \emptyset\}. 
\end{equation*}
\begin{definition}[Definition 3.1 in~\cite{mm22}]
    \label{definition:convex-general-pos}
    Assume~\eqref{eq:union_is_compact}. The sets from $\sL$ are said to be in general position if, for each closed convex subset $A$ of each exposed face of $\ch(\sL)$, the family $\sM(\sL,\, A)$ is finite, and
    \begin{equation*}
        \sum_{L \in \sM(\sL,\, A)} (1  + \dim(A \cap L)) \leq \dim(A) + 1.
    \end{equation*}
\end{definition}

Note that the above definition applies to general convex sets whose algebraic convex hull is compact. It turns out that, for strictly convex sets, the definition can be replaced by much more practically useful and simple one, see 
\begin{proposition}[Proposition 3.6 in~\cite{mm22}]
    \label{prop:strictly-convex-general-pos}
    If all sets in $\sL$ are strictly convex, Definition~\ref{definition:convex-general-pos} is equivalent to any of the following:
    \begin{enumerate}
        \item For all $m=0, \dots, d-1$, and each $m$-dimensional face $F_m$ of $\ch(\sL)$, the family $\sM(\sL, F_m)$ is finite and
        \begin{equation*}
            \sum_{L \in \sM(\sL, F_m)} (1 + \dim(F_m \cap L)) = m + 1.
        \end{equation*}
        \item For all $m = 0, \dots, d -1$, and each $m$-dimensional face $F_m$ of $\ch(\sL)$, exactly $m+1$ sets among $\sL$ intersect $F_m$.
        \item For all $m = 0, \dots, d -1$, and each $m$-dimensional exposed face $F_m$ of $\ch(\sL)$, exactly $m+1$ sets among $\sL$ intersect $F_m$.
    \end{enumerate}
\end{proposition}

As noted in the introduction, the central concept underlying recursive peeling is that of a contributing set. We now give its formal definition.

\begin{definition}
A set $L \in \sL$ is called a \emph{contributing} set to $\sL$ if $\ch(\sL) \neq \ch(\sL \setminus \{L\})$.  Equivalently, deleting $L$ from the family strictly shrinks the convex hull generated by the family. We also define
$$
\fC(\sL):=\{L \in \sL:\ch(\sL) \neq \ch(\sL \setminus \{L\})\}
$$
to be the set of contributing sets to $\sL$.
\end{definition}

Now we shall define two types of peelings of $\ch(\sL)$. 
\begin{definition}
    \label{m-point-peeling-def}
For an integer $m$ such that $0 \leq m < |I|\leq\infty$, the \emph{$m$-point peeling} of the family $\sL$ (or of the convex hull $\ch(\sL)$) is the set $\ch^{[m]}(\sL) = \bigcap_{J\subset I,|J|=m} \cch\left(\bigcup_{i \in I \setminus J} L_i\right)$.
\end{definition}

Thus, if $|I|$ is finite, then \(\ch^{[m]}(\sL)\) consists of all points which belong to the closed convex hull of every subfamily of $\sL$ of cardinality $|I|-m$. Alternatively, $\ch^{[m]}(\sL)$ consists of points that remain in the closed convex hull generated by the family no matter which $m$ sets are removed. Observe also that the closure on the right-hand side is essential, for otherwise $\ch\left(\bigcup_{i \in I \setminus J} L_i\right)$ might not be closed even under assumption~\eqref{eq:union_is_compact}.

The second type of peeling generalizes that of Calka and Quilan \cite{calka-simple-poly} and is defined as follows.
 
\begin{definition}
    \label{convex-hull-peeling}
Assume~\eqref{eq:union_is_compact}. Set $I_1:=I$ and $\sL^{(1)}:=\sL$ and put $d_1:={\rm dim}\,\ch(\sL)$. Suppose that $I_m$ and $\sL^{(m)}=\{L_i:i\in I_m\}$ have been defined. Put
$$
\sC^{(m)}(\sL):=
\left\{
i\in I_m:
\ch\left(\bigcup_{k\in I_m}L_k\right)
\neq
\ch\left(\bigcup_{k\in I_m\setminus\{i\}}L_k\right)
\right\}.
$$
Thus, $\sC^{(m)}(\sL)$ is the set of indices of sets contributing to $\sL^{(m)}$. If ${\rm dim}\,\ch(\sL^{(m)})<d_1$ or $\sC^{(m)}(\sL)=\varnothing$, the recursive peeling terminates. Otherwise, put
$$
I_{m+1}:=I_m\setminus \sC^{(m)}(\sL),
\qquad
\sL^{(m+1)}:=\{L_i:i\in I_{m+1}\}.
$$
The $m$-th recursive convex hull peeling is defined by
$$
\ch_{[m]}(\sL):=\cch(\sL^{(m)})
=
\cch\left(\bigcup_{i\in I_m}L_i\right),\quad m<D(\sL),
$$
where $D(\sL)$ is the first terminal index given by
$$
D(\sL):=\inf\{m\in\mathbb{N}:{\rm dim}\,\ch(\sL^{(m)})<d_1\text{ or }\sC^{(m)}(\sL)=\varnothing\}.
$$
Here we stipulate $\inf\varnothing=+\infty$. If $D(\sL)<\infty$, then the number of non-terminal recursive peeling layers is $D(\sL)-1$. We also make a convention that $\sC^{(m)}(\sL)=\varnothing$ for $m\geq D(\sL)$.
\end{definition}

We shall use convention that $\sC^{(1)}(\sL)=\sC(\sL)$ and observe also that 
$$
\fC(\sL):=\{L_i\in \sL:i\in\sC^{(1)}(\sL)\}\text{ and, more generally, }\fC(\sL^{(m)}):=\{L_i\in \sL:i\in\sC^{(m)}(\sL)\}.
$$

Note that when the family $\sL$ consists of a finite number of singletons, the recursive convex hull peeling of $\sL$ coincides with the classical convex hull peeling of a point configuration considered by Calka and Quilan~\cite{calka-simple-poly}. In what follows we use superscripts $[m]$ for the $m$-point peeling and subscripts $[m]$ for the recursive peeling.

At this point, a remark and a couple of illustrative examples are in order. Although Definition~\ref{convex-hull-peeling} is rather intuitive for finite families of sets, it may lead to unexpected phenomena when the family $\sL$ is infinite. Much of the bookkeeping introduced above (and to be introduced further) concerning the distinction between closed and algebraic convex hulls, the compactness of $\ch(\sL)$ and of the convex hulls of its subfamilies, and the general position of \(\sL\) and its subfamilies is intended precisely to rule out such degeneracies.

\begin{example}[Uncountable family of singletons]\label{ex:bad_behavior1}
Let $\sL:=\big\{\{(x,y)\}: x^2+y^2\leq 1\big\}$ be the collection of all singletons in the closed unit disk. Then 
$D(\sL)=3$, while $\fC(\sL^{(1)})=\big\{\{(x,y)\}: x^2+y^2=1\big\}$ is the collection of singletons lying on the unit circle, and $\sL^{(2)}=\big\{\{(x,y)\}: x^2+y^2<1\big\}$ is the collection of singletons in the open unit disk. However, no singleton in $\sL^{(2)}$ is contributing to $\sL^{(2)}$, since each such point lies in a segment contained in the open unit disk. This demonstrates the importance of the compactness assumption.
\end{example}

\begin{example}[The role of the general position concept]\label{ex:bad_behavior2}
Consider a family $\sL$ consisting of four sets: three singletons forming a triangle, and a convex set which touches one of the edges of this triangle while remaining inside it. This fourth set is not contributing to $\sL$, although it intersects the boundary of the convex hull. The obstruction is that the four sets are not in general position, see Lemma~\ref{lem:noncontrib-interior} below.

By its nature, the peeling procedure involves subfamilies of the original family $\sL$. Thus, general position should also be required for the relevant subfamilies. Since this concept is available only for families whose convex hulls are compact, compactness of the convex hulls of subfamilies must also be imposed; for infinite families $\sL$, it does not follow automatically from the compactness of $\ch(\sL)$. In conjunction with Example~\ref{ex:bad_behavior1}, this provides further evidence for the importance of the compactness assumption.
\end{example}

\subsection{Geometric lemmas on contributing sets}

The next two lemmas will formalize the intuitive idea that contributing sets must touch the boundary of the convex hull, and the rest of the sets should be in the topological interior of the hull. Whereas the former holds true under a very mild assumption that $\ch(\sL)$ is bounded, the latter requires compactness and the general position.
\begin{lemma}
        \label{lemma:cont-set-intersect-boundary}
        Let $\sL = \{ L_i, i \in I \}$ be a family of sets from $\sK^d$ and let $L\in\fC(\sL)$ be a contributing set to $\sL$. Assume that $\ch(\sL)$ is a bounded set different from a singleton. Then, $L \cap {\rm relbd}\, (\ch(\sL)) \neq \emptyset$.
        \end{lemma}
        \begin{proof}
First, observe that convexity entails 
$$
{\rm relint}\, (\ch(\sL))={\rm relint}\, (\cch(\sL))\quad\text{and}\quad {\rm relbd}\, (\ch(\sL))={\rm relbd}\, (\cch(\sL)).
$$ 
We shall prove by contradiction. Assume $L \cap {\rm relbd}\, (\cch(\sL)) = \varnothing$. This means that $L \subseteq {\rm relint}(\cch(\sL))$.
             By the Krein-Milman theorem, a compact convex set $\cch(\sL)$ is equal to the convex hull of its extreme points:
            \begin{equation}
                \label{k-m-equality}
                \cch\left(\sL\right) = \ch\left(\ext\left(\cch\left(\sL\right)\right)\right).
            \end{equation}
            A fundamental property of the extreme points is that they must lie on the relative boundary of the set:
            \begin{equation*}
                \ext\left(\cch\left(\sL\right)\right) \subseteq {\rm relbd}\,(\cch(\sL)).
            \end{equation*}
            Since $L \subseteq {\rm relint}\,(\cch(\sL))$, it follows that $L$ cannot contain any extreme points of $\cch(\sL)$. Assume that we have proved
            \begin{equation}\label{eq:lem-2-1_proof1}
                \ext\left(\cch\left(\sL\right)\right) \subseteq \cch\left(\bigcup_{J \in \sL \setminus \{L\}} J\right).
            \end{equation}
            Taking the convex hull of both sides and using \eqref{k-m-equality}:
            \begin{equation*}
                \cch(\sL) = \ch(\ext(\cch(\sL))) \subseteq \cch\left(\bigcup_{J \in \sL \setminus \{L\}} J\right) = \cch(\sL \setminus \{L\}).
            \end{equation*}
            Therefore, 
            $$
            L\subseteq {\rm relint}(\cch(\sL)) = {\rm relint}(\cch(\sL \setminus \{L\}))={\rm relint}(\ch(\sL \setminus \{L\}))\subseteq \ch(\sL \setminus \{L\}).
            $$
            Thus, $L$ is not contributing. 
            
            We shall now verify~\eqref{eq:lem-2-1_proof1}. Let $x\in\ext(\cch(\sL))$. There exists a sequence
$$
x_n\in\ch\left(L\cup \bigcup_{J\in\mathcal L\setminus\{L\}}J\right),\quad n\in\mathbb{N},
$$
such that $x_n\to x$. Write
$$
   x_n=\lambda_n y_n+(1-\lambda_n)z_n,\quad \lambda_n\in[0,1],\quad
   y_n\in L,\quad
   z_n\in
   \ch\left(
      \bigcup_{J\in\mathcal L\setminus\{L\}}J
   \right).
$$
Since $L$ is compact and $(\lambda_n)\subset[0,1]$, after passing to
a subsequence we may assume that
$$
   \lambda_n\to\lambda,
   \quad
   y_n\to y\in L .
$$
Moreover, since $\ch(\sL)$ is bounded, the sequence $(z_n)$ is bounded. Passing to a further subsequence, we may
assume that
$$
   z_n\to z \in \cch\left(\bigcup_{J\in\mathcal L\setminus\{L\}}J\right).
$$
Consequently, $x=\lambda y+(1-\lambda)z$. Observe that $y\in L\subset{\rm relint}\,(\cch(\sL))$, whereas 
$x\in\ext (\cch(\sL))\subseteq {\rm relbd}\,(\cch(\sL))$. Hence $x\ne y$. If $\lambda>0$, then the representation
$x=\lambda y+(1-\lambda)z$ writes the extreme point $x$ as a non-trivial convex combination of two
points of $\cch(\sL)$, which is impossible. Therefore, $\lambda=0$, and hence
$$
   x=z\in \cch\left(\bigcup_{J\in\mathcal L\setminus\{L\}}J\right).
$$
This proves~\eqref{eq:lem-2-1_proof1}. The proof of lemma is complete.
\end{proof}

\begin{lemma}\label{lem:noncontrib-interior}
Let $\sL=\{L_i,\ i\in I\}$ be a family of sets in $\sK^d$ satisfying~\eqref{eq:union_is_compact} and which are also in general position. If
$i\notin \sC(\sL)$ (equivalently, $L_i\notin \fC(\sL)$), then 
$$
L_i \subseteq {\rm relint}\,\left(\ch\left(\bigcup_{j\in I\setminus\{i\}} L_j\right)\right)={\rm relint}\,\left(
    \ch\left(\bigcup_{j\in I} L_j\right)\right).
$$
\end{lemma}

\begin{proof}
Set $C:=\ch\left(\bigcup_{j\in I\setminus\{i\}}L_j\right)$ and observe that also $C=\ch(\sL)$. All topological notions in the proof should be understood relative to
${\rm aff}\,C$. Thus, replacing the ambient space by ${\rm aff}\,C$, we may assume without loss of generality that $C$ is full-dimensional. Since $i\notin \sC(\sL)$, removing $L_i$ does not change the convex hull, hence $\ch\left(\bigcup_{j\in I} L_j\right)=C$. In particular,
$L_i\subset C$. Assume that $L_i\cap \partial C\neq \varnothing$ and pick a point $x\in L_i\cap \partial C$. 
Let $H$ be a supporting hyperplane of $C$ in ${\rm aff}\,C$ at $x$, and let $F:=C\cap H$ be the exposed face of $C$ determined by $H$. Define $J:=\{j\in I\setminus\{i\}: L_j\cap F\neq \varnothing\}$ and note that $J$ is a finite set by the general position assumption. Then, 
\begin{equation*}
    F = \ch\left(\bigcup_{j\in J}(L_j\cap F)\right),    
\end{equation*}
see Eq.~(3.1) in~\cite{mm22}.

Now use the dimension estimate: if $m\in\mathbb{N}$ and $B_1,\dots,B_m$ are convex sets, then
\begin{equation*}
    \dim\ch\left(\bigcup_{r=1}^m B_r\right)
    \le \sum_{r=1}^m (\dim B_r+1)-1.
\end{equation*}
Applying this to the family $\{L_j\cap F\}_{j\in J}$ and using the representation above, we get
\begin{equation}\label{eq:lem-2-2-eq1}
    \dim F +1
    \le
    \sum_{j\in J}\left(1+\dim(L_j\cap F)\right).
\end{equation}
On the other hand, since $F$ is itself an exposed face of $C$, $C=\ch\left(\bigcup_{j\in I} L_j\right)$ and $x\in L_i\cap F\neq \varnothing$, the general position assumption implies that
\begin{equation}\label{eq:lem-2-2-eq2}
    \sum_{j\in J\cup \{i\}} \left(1+\dim(L_j\cap F)\right)\le \dim F+1.
\end{equation}
Combining~\eqref{eq:lem-2-2-eq1} and~\eqref{eq:lem-2-2-eq2} yields
\begin{equation*}
    1+\dim(L_i\cap F)\le 0,
\end{equation*}
which is impossible. Thus, $L_i\cap \partial C=\varnothing$. Since $L_i\subset C$, it follows that $L_i\subset \Int (C)$ in the relative topology of ${\rm aff}\,C$, that is, $L_i\subseteq {\rm relint}(C)$
in $\mathbb R^d$. This completes the proof.
\end{proof}

\begin{remark}
Example~\ref{ex:bad_behavior2} shows that the general position assumption (and hence the compactness of $\ch(\sL)$) is essential in Lemma~\ref{lem:noncontrib-interior}.
\end{remark}

 \begin{lemma}\label{lem:increase_family}
   Let $\sL_1\subseteq \sL_2\subseteq\sL$ be two families of sets in $\mathcal{K}^d$ such that $\ch(\sL_1)=\ch(\sL_2)$. Assume that $\sL_2$ satisfies~\eqref{eq:union_is_compact} and is in general 
   position. Then $\fC(\sL_1)=\fC(\sL_2)$.
   \end{lemma}
   \begin{proof}
Denote $C=\ch(\sL_1)=\ch(\sL_2)$. First, if $L\in \sL_2\setminus \sL_1$, then $\sL_1\subseteq \sL_2\setminus\{L\}$ and hence
$C=\ch(\sL_1)\subseteq \ch(\sL_2\setminus\{L\})\subseteq \ch(\sL_2)=C$. Thus, $L\notin \fC(\mathcal L_2)$. This rules out the possibility that some $L\in\sL_2\setminus\sL_1$ contributes to $\sL_2$.

If $L\in\sL_1$ but $L\notin \fC(\sL_1)$, then $\ch(\sL_1\setminus\{L\})=C$. Using $\sL_1\setminus\{L\}\subseteq \sL_2\setminus\{L\}\), we get $C\subseteq \ch(\sL_2\setminus\{L\})\subseteq \ch(\mathcal L_2)=C$. Hence $L\notin \fC(\sL_2)$. This proves
$$
\fC(\sL_2)\subseteq \fC(\sL_1).
$$
Conversely, suppose that $L\in\sL_1$ but $L\notin \fC(\sL_2)\). By Lemma~\ref{lem:noncontrib-interior}, applied to the general-position family $\sL_2$,
$$
L\subseteq {\rm relint}\bigl(\ch(\mathcal L_2\setminus\{L\})\bigr)={\rm relint}\,C.
$$
Therefore $L$ does not hit the relative boundary of $C=\ch(\sL_1)$, whence can not be a contributing set to $\sL_1$ by Lemma~\ref{lemma:cont-set-intersect-boundary}. Thus, $L\notin\fC(\sL_1)$.
\end{proof}

    The following lemma can be thought of as a variant of the classic Krein-Milman theorem, but for contributing sets. 
    \begin{lemma}
        \label{lemma:ch-is-ch-of-contrib}
         Let $\sL=\{L_i,\ i\in I\}$ be a family of sets in $\sK^d$ satisfying~\eqref{eq:union_is_compact} and which are also in general position. Then
         \begin{equation*}
             \ch(\sL) = \ch \left(\bigcup_{i \in \sC(\sL)} L_i \right).
         \end{equation*}
     \end{lemma}
\begin{proof}
Let $C = \ch(\sL)$ and let $C_{\sC} = \ch(\bigcup_{i \in \sC(\sL)} L_i)$. The inclusion $C_{\sC} \subseteq C$ is obvious. We must show the reverse inclusion, $C \subseteq C_{\sC}$. As was explained in the previous proof we may assume that $C$ is full-dimensional.
    
By Lemma~\ref{lem:noncontrib-interior} if $i\notin \mathcal{C}(\sL)$, then 
$L_i\subset \Int(C)$. In particular, no non-contributing set meets any proper exposed face of $C$, because every exposed face is contained in $\partial C$. Take now any exposed face $F$ of $C$. As we have observed,
$$
L_i\in \mathcal{M}(\sL,F)~\Longrightarrow~i\in \sC(\sL),
$$
whence
$$
F=\ch\left(\bigcup_{L\in\mathcal{M}(\sL,F)}(F\cap L)\right)\subseteq \ch\left(\bigcup_{i \in \sC(\sL)}(F\cap L_i)\right)\subseteq \ch\left(\bigcup_{i \in \sC(\sL)}L_i\right).
$$
Thus, every proper exposed face of $C$ is contained in $C_{\sC}$. In particular, every extreme point of $C$ is contained in $C_{\sC}$, see p.~75 in~\cite{schneider}. By the Krein-Milman theorem, this yields $C\subseteq C_{\sC}$. The proof is complete.
\end{proof}

\begin{example}
Take $d=1$, $I=\mathbb{N}$ and $L_i:=\{i\}$, $i\in\mathbb{N}$. Then $C=\ch(\sL)=[1,+\infty)$ is not compact. Although the corresponding dimension condition from the definition of general position is satisfied, the conclusion of Lemma~\ref{lemma:ch-is-ch-of-contrib} fails: $\sC(\sL)=\{1\}$, while $C_{\sC}=\{1\}\neq C$. 

\end{example}

The next lemma verifies an intuitive fact: the convex hull obtained after removing one contributing set can be represented as a convex hull of the rest of contributing elements together with contributing sets of the next layer.
However, from now on we shall work intensively also with subfamilies of $\sL$ obtained by removing some members. Thus, we need to impose compactness and general position assumption also on subfamilies of $\sL$. The easiest way is to impose the restriction on all subsets. This condition is too strong in many applications. For this reason we shall introduce the following

\begin{definition}
A family $\sL=\{L_i:i\in I\}$ of subsets of $\sK^d$ satisfying~\eqref{eq:union_is_compact} is said to be in general position {\em cofinitely} if, for every finite set $F\subset I$, the family $\sL_F:=\{L_i:i\in I\setminus F\}$ is in general position. In particular, then $\ch(\sL_F)$ is compact for every finite $F\subset I$.
\end{definition}

\begin{lemma}\label{lemma:remove-1-point-from-ch}
Let $\sL := \{ L_i, i \in I\}$ be a family of sets from $\sK^d$ which satisfies~\eqref{eq:union_is_compact} and $D(\sL)\geq 3$. Suppose further, that either every subfamily of $\sL$ is in general position or $\sL$ is in general position cofinitely and $|\sC^{(1)}(\sL)|<+\infty$. Then, for every $i\in\sC(\sL)$, we have
\begin{equation*}
\ch\left(\bigcup_{j \in I \setminus \{i \}} L_j\right) = \ch\left(\bigcup_{j \in \left( \sC^{(1)}(\sL) \cup \sC^{(2)}(\sL)\right) \setminus \{ i\} } L_j\right).
\end{equation*}
\end{lemma}
\begin{proof}
Fix $i \in \sC(\sL)$. By Lemma~\ref{lemma:ch-is-ch-of-contrib}, it suffices to show that if $L_j$ is contributing to $\sL\setminus \{L_i\}$, then $j\in \sC^{(1)}(\sL)\cup\sC^{(2)}(\sL)$. We argue by contradiction. Suppose that $j\not\in \sC^{(1)}(\sL)\cup\sC^{(2)}(\sL)$. Since all sets from $\sL^{(2)}$ are in general position, hence their convex hull is compact, Lemma~\ref{lem:noncontrib-interior} implies that $L_j \subseteq {\rm relint}(\ch_{[2]}(\sL))$. Note that $(\bigcup_{i \in I \setminus  \sC^{(1)}(\sL) } L_i) \subseteq (\bigcup_{j \in I \setminus \{i\}} L_j)$. Consequently,
\begin{equation*}
\ch\left(\bigcup_{k \in I \setminus \sC^{(1)}(\sL)} L_k\right) = \ch_{[2]}\left(\sL\right) \subseteq \ch\left(\bigcup_{k \in I \setminus \{ i\}} L_k\right).
\end{equation*}
Since $D(\sL)\geq 3$ the dimensions coincide and thereupon $L_j \subseteq {\rm relint}(\ch(\bigcup_{k \in I \setminus \{ i\} } L_k))$. Hence it cannot be a contributing set to $\sL\setminus \{L_i\}$ by Lemma~\ref{lemma:cont-set-intersect-boundary}. This contradiction proves the lemma.
\end{proof}

The next lemma generalizes the above fact to the first $m$ layers. Essentially, its proof is an inductive application of Lemma~\ref{lemma:remove-1-point-from-ch}. The proof will be given in the Appendix.
\begin{lemma}\label{lem:deterministic-layer-reduction}
Let $\sL := \{ L_i, i \in I\}$ be a family of  sets from $\sK^d$ satisfying~\eqref{eq:union_is_compact}. Fix $0<m<|I|$ and let $J_m\subset I$ be a fixed subset of cardinality $m$, $|J_m|=m$.
Assume that either every subfamily of $\sL$ is in general position or $\sL$ is general position cofinitely and $|\sC^{(r)}(\sL)|<+\infty$ for $1\leq r\leq m+1$. Suppose that $D(\mathcal L)\ge m+2$. Then
    \begin{equation*}
        \ch\left(\bigcup_{j\in I\setminus J_m} L_j\right)
        =
        \ch\left(
        \bigcup_{j\in \left(\bigcup_{r=1}^{m+1} \sC^{(r)}(\sL)\right)\setminus J_m} L_j
        \right).
    \end{equation*}
    In other words, after deleting any $m$ sets, the resulting convex hull is completely
    determined by the first $m+1$ peeling layers of the original family.
    \end{lemma}

\begin{lemma}\label{lem:recursive-peeling-contained-after-deletions}
Let $\sL=\{L_i:i\in I\}$ be a family of sets from $\sK^d$ satisfying~\eqref{eq:union_is_compact}. Fix $m\in\mathbb N$ and suppose that $D(\sL)\ge m+2$. Assume further that either every subfamily of $\sL$ is in general position or $\sL$ is general position cofinitely and $|\sC^{(r)}(\sL)|<+\infty$ for $1\leq r\leq m+1$. Suppose also that $\sL^{(r)}$ satisfies~\eqref{eq:union_is_compact}, for each $1\leq r\leq  m+1$. Then, for every $J\subset I$ with $|J|=m$,
$$
\ch_{[m+1]}(\sL)\subseteq \ch\left(\bigcup_{i\in I\setminus J}L_i\right).
$$
In particular, ${\rm relint}\,\bigl(\ch_{[m+1]}(\sL)\bigr)\subseteq {\rm relint}\,\ch\left(\bigcup_{i\in I\setminus J}L_i\right)$.
\end{lemma}
\begin{proof}
Since $D(\sL)\ge m+2$, the recursive peeling has not terminated at any of the steps $1,\ldots,m+1$. Hence the layers
$$
\sC^{(1)}(\sL),\ldots,\sC^{(m+1)}(\sL)
$$
are non-empty and pairwise disjoint. Since $|J|=m$, there exists $r\in\{1,\ldots,m+1\}$ such that
$$
J\cap \sC^{(r)}(\sL)=\varnothing .
$$
Indeed, otherwise the set $J$ would have to contain at least one index from each of the $m+1$ pairwise disjoint non-empty layers. By the definition of the recursive peeling, $\sL^{(r)}=\{L_i:i\in I_r\}$ is a subfamily of $\sL$. It satisfies~\eqref{eq:union_is_compact}, and it is in general position by assumption. Therefore Lemma~\ref{lemma:ch-is-ch-of-contrib}, applied to $\sL^{(r)}$, gives
$$
\ch_{[r]}(\sL)=\ch(\sL^{(r)})=\ch\left(\bigcup_{i\in \sC^{(r)}(\sL)}L_i\right).
$$
Since $\sC^{(r)}(\sL)\cap J=\varnothing$, the last set is contained in $\ch\left(\bigcup_{i\in I\setminus J}L_i\right)$. On the other hand, the recursive peelings are nested, and $r\le m+1$. Hence
$$
\ch_{[m+1]}(\sL)\subseteq \ch_{[r]}(\sL).
$$
Combining the two inclusions yields $\ch_{[m+1]}(\sL)\subseteq \ch\left(\bigcup_{i\in I\setminus J}L_i\right)$. The assertion for relative interiors follows immediately. This proves the lemma.
\end{proof}

\subsection{Wrappings for intersections of convex sets}
Both peeling operations introduced above for convex hulls of {\em unions} of convex sets
have natural dual counterparts, that we call wrappings, for {\em intersections} of convex 
sets. This duality is provided by polarity. Recall that, for a set $A\subset \mathbb R^d$, its polar is
$$
A^{o}:=\{y\in\mathbb R^d:\sup_{x\in A}\langle x,y\rangle\leq 1\}.
$$
On the class of closed convex sets containing the origin, polarity is an involution. Denote this class by $\mathfrak{K}^d_0$ and the subclass of its sets containing the origin in the interior by $\mathfrak{K}^d_{(0)}$. 
It is easy to see that $A^{o}\in \sK_{0}^d$ if and only if $A\in \mathfrak{K}^d_{(0)}$. That is, the image $A^{o}$ is compact if and only if the origin lies in the interior. In particular, the restriction of polarity on $\sK_{(0)}^d$ is a bijection. 

The map $A\mapsto A^{o}$ satisfies
$$
\left(\bigcap_{i\in I}L_i\right)^{o}
=
\cch
\left(
\bigcup_{i\in I}L_i^{o}
\right),
\quad L_i\in\fK_{(0)}^d.
$$
Thus polar duality turns intersections into closed convex hulls of unions. Since
polarity reverses inclusions, a peeling operation on the dual side corresponds
to an enlargement of the original intersection. We call this enlargement a
{\em wrapping}.

\begin{definition}\label{m-point-wrapping-def}
Let $\mathcal L=\{L_i:i\in I\}\subseteq\mathfrak{K}^d_{(0)}$, and let $m$ be an
integer such that $0\leq m<|I|$. The $m$-point wrapping of the family
$\mathcal L$, or of the represented intersection $\bigcap_{i\in I}L_i$, is
defined by
$$
\wrap^{[m]}(\mathcal L)
:=
\cch
\left(
\bigcup_{\substack{J\subset I\\ |J|=m}}
\bigcap_{i\in I\setminus J}L_i
\right).
$$
Equivalently, under polarity,
$$
\wrap^{[m]}(\mathcal L)
=
\left(
\bigcap_{\substack{J\subset I\\ |J|=m}}
\cch
\left(
\bigcup_{i\in I\setminus J}L_i^{o}
\right)
\right)^{o}.
$$
\end{definition}
Thus, whereas the $m$-point peeling consists of the points that remain in the
closed convex hull after every deletion of $m$ members of the family, the $m$-point
wrapping is the smallest closed convex set containing all points which may enter
the intersection after some deletion of $m$ members of the family.

To introduce the $m$-th recursive wrapping, observe that $L_j^{o}$ is contributing for the {\em closed} convex hull generated by the polar family $\sL^{o}:=\{L_i^{o}:i\in I\}$ if and only if
$$
L_j^{o} \not\subseteq \cch\left(\bigcup_{i\in I\setminus\{j\}}L_i^{o}\right).
$$
By polarity, this is equivalent to
$$
\bigcap_{i\in I\setminus\{j\}}L_i \not\subseteq L_j.
$$
Thus, $L_j^{o}$ is contributing on the dual side precisely when deleting $L_j$ strictly enlarges the original intersection $ \bigcap_{i\in I}L_i$.

\begin{definition}\label{m-recursive-wrapping-def}

Let $\sL=\{L_i:i\in I\}\subseteq \mathfrak{K}^d_{(0)}$, and suppose that the polar family
$$
\sL^{o}:=\{L_i^{o}:i\in I\}
$$
satisfies the assumptions of Definition~\ref{convex-hull-peeling}. Let $D(\sL^{o})$ be the depth of the recursive convex hull peeling of $\sL^{o}$. For $m<D(\sL^{o})$, let $I_m^{o}$ be the index set obtained at the $m$-th step of Definition~\ref{convex-hull-peeling} applied to the family $\sL^{o}$. 
The $m$-th recursive wrapping of $\sL$, or of the represented intersection $\bigcap_{i\in I}L_i$, is defined by
\begin{equation}\label{eq:wrapping_duality}
\wrap_{[m]}(\sL):=\left(\ch_{[m]}(\sL^{o})\right)^{o} .
\end{equation}
Equivalently,
$$
\wrap_{[m]}(\sL)=\bigcap_{i\in I_m^{o}}L_i .
$$
Here and in other similar places we use the convention that intersections over empty index sets coincide with $\mathbb R^d$.
\end{definition}

\begin{remark}
The distinction between $\cch$ and $\ch$ in the definitions of this section becomes irrelevant whenever the index set $I$ is finite. This is the case we are mostly interested in throughout the paper, and in particular in the next section devoted to $K$-convexity.
\end{remark}
        
\section{Applications to $K$-convexity}
\label{main-results}
We shall now apply the above notions to the $K$-hulls of random samples as studied in~\cite{mm22}. Let $K$ be a fixed convex body from $\sK^d_{(0)}$. Recall from \cite{mm22} that, for a set $A\subset\R^d$, its $K$-hull is
\begin{equation*}
    \ch_K(A):=\bigcap_{x\in\R^d:\,A\subset K+x}(K+x).
\end{equation*}
Note that if $K$ is the Euclidean unit ball, then $\ch_K(A)$ is the ball hull of $A$. Let $\Xi_n=\{\xi_1,\dots,\xi_n\}$ be a random sample of independent points uniformly distributed in $K$, and define
\begin{equation*}
    X_n:=\bigcap_{i=1}^n (K-\xi_i)=K\ominus \Xi_n=K\ominus \ch_K(\Xi_n),
\end{equation*}
where $A \ominus B:= \{ x \in \R^d : x + B \subseteq A \}$ is the Minkowski difference. Here the second equality is a consequence of Proposition 2.2 in~\cite{mm22}. Thus, the set $X_n$ measures the difference between the `mother body' $K$ and the $K$-hull of a random sample from $K$. It is clear that as $n\to\infty$, $X_n$ converges to $\{0\}$ almost surely. In~\cite{mm22} it was proved that the rescaled set $nX_n$ converges in distribution to a limiting random convex body, as $n\to\infty$. The limiting body is defined through a Poisson hyperplane tessellation. Let $S^{d-1}$ denote the unit sphere in $\R^d$. Let $\sP_K=\{(t_i,u_i),\ i\ge 1\}$ be a Poisson process on $(0,\infty)\times S^{d-1}$ with intensity measure equal to the product of a constant $V_d(K)^{-1}$, Lebesgue measure on $(0,\infty)$, and the surface area measure $S_{d-1}(K,\cdot)$ of the body $K$. The associated half-spaces $H^-_{u_i}(t_i)$ generate a Poisson tessellation of $\R^d$, and its zero cell is
\begin{equation*}
    Z:=\bigcap_{(t_i,u_i)\in \sP_K} H^-_{u_i}(t_i), 
\end{equation*}
where $H_u^-(t) := \{ x \in \R^d: \langle x, u \rangle \leq t \}$. Theorem 5.1 in~\cite{mm22} establishes the convergence
\begin{equation}\label{eq:mm22_claim1}
n X_n~\dodn~Z,\quad n\to\infty
\end{equation}
on $\mathcal{K}^d_{(0)}$ endowed with the Hausdorff topology. The continuity of the polarity, see Lemma 7.2(iv) in~\cite{kabluchko2025generalised}, yields
\begin{equation}\label{eq:mm22_claim2}
n^{-1} X_n^{o}~\dodn~Z^{o},\quad n\to\infty.
\end{equation}
Observe that $Z^{o}=\ch(\Pi_K)$, where $\Pi_K$ is the Poisson point process on $\mathbb{R}^d$ obtained as the image of $\sP_K$ under $(t,u)\mapsto t^{-1}u$. With probability one both $Z$ and $Z^{o}$ are polytopes and contain the origin in the interior. In particular, one can also write $Z^{o}=\ch(\cup_{x\in \Pi_K}[0,x])$.

Our aim is to apply the peelings and wrappings procedures introduced above to the set $X_n$ and derive distributional limit theorems for these derived sets. 

In theorems below we also assume that $K$ is strictly convex and regular. The latter means that the normal cone $N(K,x)$ is one-dimensional for all $x\in\partial K$.

\begin{theorem}[The $m$-th recursive peeling and wrapping]
\label{thm:recursive-peeling-convergence-random}
Assume that the convex body $K$ is strictly convex and regular. For every $m\in\NN$, the random vector of convex bodies $\left(n^{-1}(X^o_n)_{[1]},\dots,n^{-1}(X^o_n)_{[m]}\right)$
converges in distribution, as $n\to\infty$, on the space $(\sK^d_{(0)})^m$ endowed with the product topology to
$\left(Z^o_{[1]},\dots,Z^o_{[m]}\right)$. Moreover, $\left(\wrap_{[1]}(n X_n),\ldots,\wrap_{[m]}(n X_n)\right)$ converges in distribution to $(\wrap_{[1]}(Z),\ldots,\wrap_{[m]}(Z))$, as $n\to\infty$.
\end{theorem}

\begin{theorem}[The $m$-point peeling and wrapping]
\label{thm:m-point-peeling-convergence-random}
Assume that the convex body $K$ is strictly convex and regular. Then, for every $m\in\NN$, the random vector of convex bodies
$\left(n^{-1}(X_n^o)^{[1]},\dots,n^{-1}(X_n^o)^{[m]}\right)$ converges in distribution, as $n\to\infty$, on the space $(\sK^d_{(0)})^m$ endowed with the product topology to
$\left((Z^o)^{[1]},\dots,(Z^o)^{[m]}\right)$. Moreover, $\left(\wrap^{[1]}(n X_n),\ldots,\wrap^{[m]}(n X_n)\right)$ converges in distribution to $(\wrap^{[1]}(Z),\ldots,\wrap^{[m]}(Z))$, as $n\to\infty$.
\end{theorem}

\section{Proofs}\label{sec:proof}
Theorems~\ref{thm:m-point-peeling-convergence-random} and~\ref{thm:recursive-peeling-convergence-random} will be derived from~\eqref{eq:mm22_claim1} and~\eqref{eq:mm22_claim2} by the continuous mapping approach, using
deterministic results which establish continuity of the mappings
$$
   \sL \mapsto \ch^{[m]}(\sL)
   \qquad\text{and}\qquad
   \sL \mapsto \ch_{[m]}(\sL)
$$
on a suitable space of point measures. Let
$$
   \sK^d_0\setminus\{\{0\}\},
$$
where $\sK^d_0$ is equipped with the Hausdorff metric $d_H$. For $A\in\sK^d_0\setminus\{\{0\}\}$, put
$$
   \rho(A):=d_H(A,\{0\})=\sup_{x\in A}\|x\|.
$$
Thus, $\rho(A)$ measures how far the compact convex set $A$, which contains the origin, reaches away from the origin. For $r>0$, define
$$
   \mathcal{B}_r
   :=
   \{A\in\sK^d_0\setminus\{\{0\}\}:\rho(A)\ge r\}
   =
   \{A\in\sK^d_0\setminus\{\{0\}\}:A\cap B_r^c(0)\neq\varnothing\}.
$$
We shall use the Hausdorff topology on $\sK^d_0\setminus\{\{0\}\}$, but the point
measures considered below are not required to be locally finite with
respect to all Hausdorff-compact subsets of $\sK^d_0\setminus\{\{0\}\}$. Instead, we
use the bornology generated by the sets $\mathcal{B}_r$, $r>0$. A Borel
measure $\mu$ on $\sK^d_0\setminus\{\{0\}\}$ is called locally finite if
$\mu(\mathcal{B}_r)<\infty$, $r>0$. Equivalently, $\mu$ has only finitely many atoms whose Hausdorff
distance from $\{0\}$ is at least $r$, for every $r>0$.

A sequence of locally finite Borel measures $(\mu_n)_{n\geq 1}$ on
$\sK^d_0\setminus\{\{0\}\}$ is said to converge vaguely or, more precisely, $\mathcal{B}$-vaguely, to a
locally finite measure $\mu$, written
$$
   \mu_n~\overset{{\rm v}}{\to}~\mu,\quad n\to\infty,
$$
if $\int f(A)\,\mu_n({\rm d}A)\to\int f(A)\,\mu({\rm d}A)$, for every bounded Hausdorff-continuous function $f:\sK^d_0\setminus\{\{0\}\}\to\mathbb{R}$ which vanishes in a Hausdorff neighbourhood of $\{0\}$. In this sense the deleted singleton $\{0\}$ is treated as
the point at infinity: atoms may accumulate only at $\{0\}$, and vague
convergence treats only those atoms which remain a positive Hausdorff
distance away from $\{0\}$.

\subsection{Continuity of peelings and wrappings with respect to the vague convergence}
For each $n\in\mathbb{N}$, let $\sA^{(n)} : =\{A_i^{(n)} \in \sK^d_0 \setminus \{0 \}, i \in I_n\}$  be a collection of compact convex sets. Suppose that, for each $n\in\mathbb{N}$, the point measure
$$
\mu_n:=\sum_{i\in I_n}\delta_{A_i^{(n)}}
$$
is locally finite and
\begin{equation}\label{thm:m-point-peeling-convergence-determ-a1}
\mu_n~\overset{{\rm v}}{\to}~\mu,\quad n\to\infty,
\end{equation}
for a locally finite point measure 
$$
\mu:=\sum_{i\in I}\delta_{A_i}.
$$
In what follows, we shall use the fact that any family $\sA$ in general position cannot contain duplicates. Consequently, in this case the corresponding point measure $\mu$ is automatically simple.

The next theorem establishes the continuity of the $m$-th recursive peeling operation.

\begin{theorem}\label{thm:eventual-equality-layers}
Fix $m\in\mathbb{N}$ and assume that~\eqref{thm:m-point-peeling-convergence-determ-a1} holds true. Suppose further that the family $\sA$ is in general position cofinitely and $0\in \Int(\ch_{[m]}(\sA))$. Then, there exists $N\in\NN$ such that for all $n\ge N$,
\begin{equation*}
    \sC^{(r)}(\sA^{(n)})=\sC^{(r)}(\sA),\qquad r=1,\dots,m-1.
\end{equation*}
and 
\begin{equation*}
    \ch_{[m]}(\sA^{(n)}) \overset{H}{\to} \ch_{[m]}(\sA),\quad n\to\infty.
\end{equation*}
\end{theorem}

The next theorem establishes continuity of the $m$-point peeling. Here we shall also assume strict convexity.    

\begin{theorem}\label{thm:m-point-peeling-convergence-determ}
Fix $m\in\mathbb{N}$ and assume that~\eqref{thm:m-point-peeling-convergence-determ-a1} holds true. Suppose further that
    \begin{enumerate}
        \item all atoms of $\mu_n$ are strictly convex and regular, for all sufficiently large $n\in\mathbb{N}$;
         \label{thm:m-point-peeling-convergence-determ-a2}
        \item the family $\sA$ is in general position cofinitely;
          \label{thm:m-point-peeling-convergence-determ-a3}
        \item $0\in \Int\ch(\sA\setminus F)$, for every finite set $F$.
         \label{thm:m-point-peeling-convergence-determ-a4}
    \end{enumerate}
Then $\ch^{[m]}(\sA^{(n)}) \overset{H}{\to} \ch^{[m]}(\sA)$, as $n\to\infty$.
\end{theorem}

\subsection{Auxiliary results: the vague convergence and contributing sets}
For the proof of Theorems~\ref{thm:eventual-equality-layers} and~\ref{thm:m-point-peeling-convergence-determ} we need a couple of auxiliary lemmas.

\begin{lemma}\label{lemma:converg-to-contrib-ch}
Suppose that~\eqref{thm:m-point-peeling-convergence-determ-a1} holds true, $0 \in \Int(\ch(\sA))$ and the family $\sA$ is in general position. Then the following holds:
\begin{equation*}
\ch \left(\bigcup_{i \in I_n} A_i^{(n)} \right) \overset{H}{\to} \ch \left(\bigcup_{i \in \sC(\sA) } A_i \right).
\end{equation*}
\end{lemma}
\begin{proof}
Since $\ch(\sA)$ contains origin in its interior, there exists $r > 0$ such that
\begin{equation}\label{lem:4-3-ball-insider}
B_{2r}(0) \subseteq \Int\left(\ch\left(\bigcup_{i \in I} A_i\right)\right).
\end{equation}
By decreasing $r$, if necessary, we may also assume that $\mu(\partial \mathcal B_r)=0$. By local finiteness of $\mu$, only finitely many atoms of $\mu$ intersect $B_r^c(0)$. Relabel them as
$A_1,\dots,A_l$. Thus, $A_i\subset B_r(0)$, for $i>l$. In particular, this implies that $\ch(\mathcal A)$ is compact.  Then, by vague convergence and the standard matching property of atoms, see~\cite[Proposition 3.13]{resnick}, the atoms of $\mu_n$ may be relabeled in such a way that
\[
A_i^{(n)}~\overset{H}{\to}~A_i,\quad n\to\infty,\quad i=1,\dots,l,
\]
and $A_i^{(n)}\subseteq B_r(0)$, for $i>l$ and all sufficiently large $n$.

By continuity of finite unions and of the convex hull operation, see Theorem~12.3.5 in~\cite{Schneider+Weil}, we have
\begin{equation}\label{conv-finite}
\ch\left(\bigcup_{i=1}^l A_i^{(n)}\right)~\overset{H}{\to}~\ch\left(\bigcup_{i=1}^l A_i\right).
\end{equation}
We now identify the limit on the right-hand side. Every contributing set $A_i$, $i\in\sC(\sA)$, meets the boundary of $\ch(\sA)$ by Lemma~\ref{lemma:cont-set-intersect-boundary} (note that $\ch(\sA)$ is full-dimensional). In view of~\eqref{lem:4-3-ball-insider} such a set must intersect $B_r^c(0)$. Hence
$\sC(\sA)\subseteq \{1,\dots,l\}$. Using Lemma~\ref{lemma:ch-is-ch-of-contrib}, we get
\begin{equation}\label{limit-contributing-only}
\ch\left(\bigcup_{i=1}^l A_i\right) = \ch\left(\bigcup_{i\in I}A_i\right) = \ch\left(\bigcup_{i\in\sC(\sA)}A_i\right).
\end{equation}
Consequently,
$$
\ch\left(\bigcup_{i=1}^l A_i^{(n)}\right)~\overset{H}{\to}~\ch\left(\bigcup_{i\in\sC(\sA)}A_i\right).
$$
By~\eqref{lem:4-3-ball-insider} and~\eqref{conv-finite} and~\eqref{limit-contributing-only}, we have
$$
B_r(0)\subseteq \ch\Big(\bigcup_{i=1}^{l}A_i^{(n)}\Big)
$$
for all sufficiently large $n$. Since every remaining atom satisfies $A_i^{(n)}\subset B_r(0)$, $i>l$, adding these atoms does not change the convex hull. Thus,
$$
\ch\Big(\bigcup_{i\in I_n}A_i^{(n)}\Big)=\ch\Big(\bigcup_{i=1}^{l}A_i^{(n)}\Big)
$$
for all sufficiently large $n$. Combining this with~\eqref{conv-finite} and~\eqref{limit-contributing-only} proves the claim.
\end{proof}

The next technical lemma formalises the intuitive fact that removing finitely many converging atoms preserves vague convergence.
\begin{lemma}\label{lemma:vague-conv-removed-atoms}
Suppose that~\eqref{thm:m-point-peeling-convergence-determ-a1} holds true and $\mu$ does not have multiple atoms. Let $J = \{i_1,\dots,i_m\} \subseteq I$ be a finite set of pairwise distinct indices. Then, after a proper relabeling of atoms of $\mu_n$, there exists $N \in \NN$ such that for all $n \geq N$, $J \subseteq I_n$ and, as $n\to\infty$,
\begin{equation*}
A_{i}^{(n)} \overset{H}{\to} A_{i},\quad n\to\infty,\quad  i \in J.
\end{equation*}
Define, for $n\geq N$,
\begin{equation*}
    \tilde\mu_n:=\sum_{i\in I_n\setminus J}\delta_{A_i^{(n)}},
    \qquad
    \tilde\mu:=\sum_{i\in I\setminus J}\delta_{A_i}.
\end{equation*}
Then
\begin{equation*}
    \tilde\mu_n \overset{{\rm v}}{\to} \tilde\mu,\quad n\to\infty.
\end{equation*}
\end{lemma}

\begin{proof}
Let \(J=\{i_1,\dots,i_m\}\subseteq I\) be fixed. Since each \(A_i\), \(i\in J\), is an atom of \(\mu\) on
\(\sK^d_0\setminus\{\{0\}\}\), we have \(A_i\neq\{0\}\). Hence, for every \(i\in J\), there exists a point
\(x_i\in A_i\) with \(\|x_i\|>0\). 
Choose $r>0$ so small that
$$
r<\min_{i\in J}\sup_{x\in A_i}\|x\|
$$
and, in addition, $\mu(\partial\mathcal{B}_r)=0$. This is possible because only countably many radii can be charged by the atoms of $\mu$. Then $A_i\cap B_r^c(0)\neq\varnothing$ for every $i\in J$.

Since $\mu$ is locally finite  only finitely many atoms of $\mu$ intersect $B_r^c(0)$. After relabeling, denote them by
$A_1,\dots,A_l$, chosen so that $J\subset \{1,\dots,l\}$.

Now apply~\cite[Proposition 3.13]{resnick} to the compact set $\mathcal{B}_r$. It follows that, after a proper
relabeling of the atoms of $\mu_n$, there exists $N\in\NN$ such that for all $n\ge N$ the atoms
\begin{equation*}
    A_1^{(n)},\dots,A_l^{(n)}
\end{equation*}
are exactly the atoms of $\mu_n$ intersecting \(B_r^c(0)\), and
\begin{equation*}
    A_i^{(n)} \overset{H}{\to} A_i, \quad i=1,\dots,l.
\end{equation*}
In particular, for all $n\ge N$, we have $J\subset I_n$ and
$$
    A_i^{(n)} \overset{H}{\to} A_i, \quad i\in J.
$$
We shall now prove the vague convergence $\tilde\mu_n \overset{{\rm v}}{\to} \tilde\mu$. Let $f:\sK^d_0\setminus\{\{0\}\}\to\R$ be a continuous function vanishing in a Hausdorff-neighborhood of $\{0\}$. Then
\begin{equation*}
    \int f\, {\rm d}\tilde{\mu}_n = \int f\, {\rm d}\mu_n-\sum_{i\in J} f(A_i^{(n)}),
\end{equation*}
and similarly
\begin{equation*}
    \int f\, {\rm d}\tilde{\mu} = \int f\, {\rm d}\mu-\sum_{i\in J} f(A_i).
\end{equation*}
Since $\mu_n \overset{{\rm v}}{\to} \mu$, we have $\int f\, d\mu_n \to \int f\, d\mu$. Moreover, for each $i\in J$, the convergence
$A_i^{(n)} \overset{H}{\to} A_i$ and continuity of $f$ imply $f(A_i^{(n)})\to f(A_i)$. Therefore,
\begin{equation*}
    \int f\, {\rm d}\tilde{\mu}_n
    \to
    \int f\, {\rm d}\mu-\sum_{i\in J} f(A_i)
    =
    \int f\, d\tilde{\mu},\quad n\to\infty.
\end{equation*}
This proves the stated vague convergence.
\end{proof}

\begin{lemma}\label{lemma:single-removed-conv}
Suppose that~\eqref{thm:m-point-peeling-convergence-determ-a1} holds true and the family $\sA$ is in general position cofinitely. Assume that $0 \in \Int(\ch_{[2]}(\sA))$. Fix $i \in \sC(\sA)=\sC^{(1)}(\sA)$. Then there exists a sequence $(A_i^{(n)})$ such that $A_i^{(n)}$ is an atom of $\mu_n$, $i\in I_n$, $A_i^{(n)} \overset{H}{\to} A_i$, and
\begin{equation}\label{lemma:single-removed-conv-claim}
\ch\left(\bigcup_{j \in I_n \setminus\{ i \}} A_j^{(n)}\right) \overset{H}{\to} \ch\left(\bigcup_{j \in \left( \sC^{(1)}(\sA) \cup \sC^{(2)}(\sA)\right) \setminus \{ i\} } A_j\right).
\end{equation}
\end{lemma}
\begin{proof}
The existence and convergence $A_i^{(n)} \overset{H}{\to} A_i$ follow by the same argument as in the proof of Lemma~\ref{lemma:vague-conv-removed-atoms} with $m=1$. Observe that $\mu$ does not have multiple atoms because $\sA$ is in general position. We only need to verify~\eqref{lemma:single-removed-conv-claim}.
For all sufficiently large $n\in\mathbb{N}$, consider the reduced families
\begin{equation*}
\sA_n^{(i)} := \{A_j^{(n)} : j \in I_n \setminus \{i\}\}\quad\text{and}\quad\sA^{(i)} := \{A_j : j \in I \setminus \{i\}\}.
\end{equation*}
The same argument as in Lemma~\ref{lemma:converg-to-contrib-ch} applies to these reduced families and yields
\begin{equation*}
\ch\left(\bigcup_{j \in I_n \setminus \{i\}} A_j^{(n)}\right)~\overset{H}{\to}~\ch\left(\bigcup_{j \in I \setminus \{i\}} A_j\right).
\end{equation*}
By Lemma~\ref{lemma:remove-1-point-from-ch} the right-hand side coincides with the right-hand side of~\eqref{lemma:single-removed-conv-claim}.
\end{proof}

Similarly as Lemma~\ref{lem:deterministic-layer-reduction} is an inductive generalization of Lemma~\ref{lemma:remove-1-point-from-ch} the next result is a generalization of Lemma~\ref{lemma:single-removed-conv} when one removes finitely many sets rather than just one.

\begin{lemma}\label{thm:converg-remove-m-points}
Suppose that~\eqref{thm:m-point-peeling-convergence-determ-a1} holds true and the family $\sA$ is in general position cofinitely. Fix $m\geq 1$ and assume $0\in \Int(\ch_{[m+1]}(\sA))$. Let $J = \{i_1,\dots,i_m\} \subseteq I$ be a finite set of pairwise distinct indices. Then, for all sufficiently large $n\in\mathbb{N}$, $J\subseteq I_n$ and, after proper relabeling of atoms, 
          \begin{equation*}
              A_{i_r}^{(n)} \overset{H}{\to} A_{i_r},\quad n\to\infty,\quad r = 1, \dots, m.
          \end{equation*}
          Furthermore,
             \begin{equation*}
                     \ch\left(\bigcup_{j \in I_n \setminus J} A_j^{(n)}\right)~\overset{H}{\to}~\ch\left(\bigcup_{j \in (\bigcup_{i = 1}^{m+1} \sC^{(i)}(\sA))  \setminus J} A_j\right),\quad n\to\infty.
             \end{equation*}
\end{lemma}
\begin{proof}
The first part follows by the same argument as in the beginning of Lemma~\ref{lemma:vague-conv-removed-atoms}. We only prove the convergence of the convex hulls. Define the reduced measures by
\begin{equation*}
    \mu_n^{(J)}:=\sum_{j\in I_n\setminus J}\delta_{A_j^{(n)}},
    \qquad
    \mu^{(J)}:=\sum_{j\in I\setminus J}\delta_{A_j}.
\end{equation*}
By Lemma~\ref{lemma:vague-conv-removed-atoms} $\mu_n^{(J)}~\overset{{\rm v}}{\to}~\mu^{(J)}$, $n\to\infty$. By Lemma~\ref{lem:recursive-peeling-contained-after-deletions}, $0\in \Int(\ch_{[m+1]}(\sA))$ implies $0\in\Int\ch(\cup_{j\in I\setminus J}A_j)$. Finally, Lemma~\ref{lemma:converg-to-contrib-ch} applied with measures $\mu_n^{(J)}$ converging to $\mu^{(J)}$ gives
$$
\ch\left(\bigcup_{j \in I_n \setminus J} A_j^{(n)}\right)~\overset{H}{\to}~\ch\left(\bigcup_{j \in \sC(\sA \setminus \{A_i:i\in J\})} A_j\right),\quad n\to\infty.
$$
It remains to note that by Lemma~\ref{lemma:ch-is-ch-of-contrib}
$$
\ch\left(\bigcup_{j \in \sC(\sA \setminus \{A_i:i\in J\})} A_j\right)=\ch\left(\bigcup_{j \in I \setminus J} A_j\right)
$$
and by Lemma~\ref{lem:deterministic-layer-reduction}
$$
\ch\left(\bigcup_{j \in I \setminus J} A_j\right)
=
\ch\left(
        \bigcup_{j\in \left(\bigcup_{r=1}^{m+1} \sC^{(r)}(\sA)\right)\setminus J} A_j
        \right).
$$
The proof is complete.
\end{proof}

\subsection{Proof of Theorem~\ref{thm:eventual-equality-layers}}
Since $0\in\Int(\ch_{[m]}(\sA))$, we have $m<D(\sA)$. Recall that
$$
\sA=\{A_i:i\in I\},
\quad\text{and}\quad
\sA^{(n)}=\{A_i^{(n)}:i\in I_n\},\quad n\in\mathbb{N}.
$$
All relabelings below are made by using the usual matching property of atoms
under vague convergence, see Proposition 3.13 in~\cite{resnick}. Since only 
finitely many atoms are matched at each stage, after increasing \(N\) finitely 
many times we keep one compatible relabeling throughout the proof.

First observe that the assumption $0\in \Int(\ch_{[m]}(\sA))$ implies that the first $m$ recursive hulls of $\sA$ are well-defined. Moreover, the recursive hulls are nested, and hence
$$
\ch_{[m]}(\sA)\subset \ch_{[r]}(\sA),
\quad r=1,\ldots,m.
$$
Consequently, $0\in \Int(\ch_{[r]}(\sA))$, for all $r=1,\ldots,m$.

If $m=1$, the assertion about the equality of the first $m-1$ layers is void. Since $0\in\Int(\ch(\sA))$, Lemma~\ref{lemma:converg-to-contrib-ch} and Lemma~\ref{lemma:ch-is-ch-of-contrib} yield
$$
\ch(\sA^{(n)})=\ch_{[1]}(\sA^{(n)})~\overset{H}{\to}~
\ch\left(\bigcup_{i\in \sC^{(1)}(\sA)}A_i\right)=\ch(\sA)=\ch_{[1]}(\sA).
$$
Thus, the theorem is proved for $m=1$. In the rest of the proof we assume $m\ge 2$.

We first prove a one-step stability assertion for cofinite subfamilies. Let
$S\subseteq I$ be finite and put
$$
\sB:=\{A_i:i\in I\setminus S\}.
$$
Assume that the associated approximating subfamilies
$$
\sB^{(n)}:=\{A_i^{(n)}:i\in I_n\setminus S\}
$$
converge vaguely to $\sB$, and assume that
$$
0\in \Int(\ch_{[2]}(\sB)).
$$
We claim that after a proper relabeling, there exists $N_{\sB}\in\mathbb N$ such
that
\begin{equation}\label{eq:thm41-proof0}
\sC^{(1)}(\sB^{(n)})=\sC^{(1)}(\sB),
\qquad n\ge N_{\sB}.
\end{equation}
Indeed, since
$$
\ch_{[2]}(\sB)\subset \ch(\sB),
$$
we have $0\in\Int(\ch(\sB))$. By Lemma~\ref{lemma:converg-to-contrib-ch} and Lemma~\ref{lemma:ch-is-ch-of-contrib},
\begin{equation}\label{eq:thm41-proof1}
\ch(\sB^{(n)})~\overset{H}{\to}~\ch(\sB).
\end{equation}
Choose \(\rho>0\) such that
$$
B_{2\rho}(0)\subset \Int(\ch(\sB))
$$
and such that the boundary of the corresponding truncation set in
$\sK_0^d\setminus\{\{0\}\}$ is not charged by the limiting point measure $\mu_{\mathcal{B}}$.
By local finiteness, only finitely many atoms of $\sB$ intersect
$B_\rho^c(0)$. Relabel them as
$$
A_1,\ldots,A_\ell.
$$
For all sufficiently large $n$, the atoms of $\sB^{(n)}$ intersecting $B_\rho^c(0)$ are precisely
$$
A_1^{(n)},\ldots,A_\ell^{(n)},
$$
and
$$
A_i^{(n)}~\overset{H}{\to}~A_i,\quad i=1,\ldots,\ell .
$$
In view of~\eqref{eq:thm41-proof1}, after increasing $N_{\sB}$, if necessary, we also have
$$
B_\rho(0)\subset \Int(\ch(\sB^{(n)})),\quad n\ge N_{\sB}.
$$
By Lemma~\ref{lemma:cont-set-intersect-boundary}, every contributing member of \(\sB\), and every contributing
member of \(\sB^{(n)}\) for \(n\ge N_{\sB}\), must meet the boundary of the
corresponding hull. Hence every such member must intersect \(B_\rho^c(0)\).
Therefore all possible contributing indices, both for \(\sB\) and for
\(\sB^{(n)}\), belong to the finite set \(\{1,\ldots,\ell\}\).

Fix $i\in\{1,\ldots,\ell\}$, and put
$$
W_i:=\ch\left(\bigcup_{j\in E\setminus\{i\}}A_j\right),
\qquad
W_i^{(n)}:=
\ch\left(\bigcup_{j\in E_n\setminus\{i\}}A_j^{(n)}\right),
$$
where $E := I \setminus S, E_n := I_n \setminus S$. We claim that
\begin{equation}\label{eq:thm41-proof2}
W_i^{(n)}~\overset{H}{\to}~W_i,\quad n\to\infty.
\end{equation}
If $i\in \sC^{(1)}(\sB)$, this follows from Lemma~\ref{lemma:single-removed-conv}, together with
Lemma~\ref{lemma:remove-1-point-from-ch}, because
$$
0\in\Int(\ch_{[2]}(\sB)).
$$
If $i\notin \sC^{(1)}(\sB)$, then
$$
W_i=\ch(\sB),
$$
so $0\in\Int(W_i)$. By Lemma~\ref{lemma:vague-conv-removed-atoms}, deleting the matched atom $W_i^{(n)}$
preserves vague convergence. Therefore, Lemma~\ref{lemma:converg-to-contrib-ch}, applied to the reduced
families $\sB\setminus\{A_i\}$ and $\sB^{(n)}\setminus\{A_i^{(n)}\}$, gives~\eqref{eq:thm41-proof2}.

We now distinguish the two cases. If \(i\in\sC^{(1)}(\sB)\), then $A_i\not\subseteq W_i$. Choose $x_i\in A_i\setminus W_i$. Since $W_i$ is compact, $\eta_i:={\rm dist}(x_i,W_i)>0$. For all sufficiently large $n$, there exists
$x_i^{(n)}\in A_i^{(n)}$ such that
$$
\|x_i^{(n)}-x_i\|<\eta_i/3,
$$
and also
$$
d_H(W_i^{(n)},W_i)<\eta_i/3.
$$
Thus, $x_i^{(n)}\notin W_i^{(n)}$, and hence $A_i^{(n)}\not\subseteq W_i^{(n)}$. Therefore, $i\in\sC^{(1)}(\sB^{(n)})$, for all sufficiently large $n$.

If $i\notin\sC^{(1)}(\sB)$, then Lemma~\ref{lem:noncontrib-interior} gives
$$
A_i\subset \Int(W_i),
$$
where the interior is the ordinary interior, since $W_i=\ch(\sB)$ is
full-dimensional. Since $A_i$ is compact, there exists $\varepsilon_i>0$
such that
$$
A_i+\varepsilon_i B_1(0)\subseteq W_i.
$$
Using
$$
A_i^{(n)}~\overset{H}{\to}~A_i
\qquad\text{and}\qquad
W_i^{(n)}~\overset{H}{\to}~W_i,\quad n\to\infty,
$$
we obtain, for all sufficiently large $n$,
$$
A_i^{(n)}\subseteq W_i^{(n)}.
$$
Thus, $i\notin\sC^{(1)}(\sB^{(n)})$, for all sufficiently large $n$.

Since the set $\{1,\ldots,\ell\}$ is finite, further increasing $N_{\sB}$, if necessary gives~\eqref{eq:thm41-proof0}. This proves the aforementioned one-step stability assertion. We now apply it inductively.  Put
$$
S_0:=\varnothing,\quad S_r:=\bigcup_{q=1}^r \sC^{(q)}(\sA),\quad 1\leq r\leq D(\sA)-1.
$$
First apply the one-step stability assertion with $\sB=\sA$. Since $0\in\Int(\ch_{[2]}(\sA))$, there exists $N_1\in\mathbb N$ such that
$$
\sC^{(1)}(\sA^{(n)})=\sC^{(1)}(\sA),\quad n\ge N_1.
$$
In particular, the first layer $\sC^{(1)}(\sA)$ is finite.

Assume now that, for some $r\in\{2,\ldots,m-1\}$, $m\leq D(\sA)-1$, we have already found
$N_{r-1}$ such that, for all $n\ge N_{r-1}$,
$$
\sC^{(q)}(\sA^{(n)})=\sC^{(q)}(\sA),\quad q=1,\ldots,r-1.
$$
Then $S_{r-1}$ is finite, and for all $n\ge N_{r-1}$ the family obtained from $\sA^{(n)}$ after deleting its first $r-1$ recursive layers is
$$
\sM^{(n)}:=\{A_i^{(n)}:i\in I_n\setminus S_{r-1}\}.
$$
The corresponding limiting family is
$$
\sM:=\{A_i:i\in I\setminus S_{r-1}\}.
$$
By Lemma~\ref{lemma:vague-conv-removed-atoms}, the point measures associated with $\sM^{(n)}$ converge
vaguely to the point measure associated with $\sM$. Moreover,
$$
\ch(\sM)=\ch_{[r]}(\sA),\quad \ch_{[2]}(\sM)=\ch_{[r+1]}(\sA).
$$
Since $r+1\le m$ and the recursive hulls are nested,
$$
\ch_{[m]}(\sA)\subset \ch_{[r+1]}(\sA)=\ch_{[2]}(\sM).
$$
Therefore, $0\in\Int(\ch_{[2]}(\sM))$. Also, every subfamily of $\sM$ is in general position cofinitely, because it is obtained from $\sA$ by deleting finitely many members. Hence the one-step stability assertion, applied to $\sM^{(n)}$ and $\sM$, gives an integer $N_r\ge N_{r-1}$ such that
$$
\sC^{(1)}(\sM^{(n)})=\sC^{(1)}(\sM),\quad n\ge N_r.
$$
By the definition of the recursive peeling, this is exactly
$$
\sC^{(r)}(\sA^{(n)})=\sC^{(r)}(\sA),\quad n\ge N_r.
$$
The induction is complete. Hence there exists $N\in\mathbb{N}$ such that
$$
\sC^{(r)}(\sA^{(n)})=\sC^{(r)}(\sA),\quad r=1,\ldots,m-1,
$$
for all $n\ge N$.

It remains to prove the Hausdorff convergence of the $m$-th recursive hulls.
For $n\ge N$, 
$$
\ch_{[m]}(\sA^{(n)})
=
\ch\left(
\bigcup_{i\in I_n\setminus S_{m-1}} A_i^{(n)}
\right).
$$
By Lemma~\ref{lemma:vague-conv-removed-atoms}, deleting the finitely many matched atoms with indices in
$S_{m-1}$ preserves vague convergence, so
$$
\sum_{i\in I_n\setminus S_{m-1}}\delta_{A_i^{(n)}}~\overset{{\rm v}}{\to}~
\sum_{i\in I\setminus S_{m-1}}\delta_{A_i},\quad n\to\infty.
$$
The limiting reduced family
$$
\sA^{(m)}=\{A_i:i\in I\setminus S_{m-1}\}
$$
is in general position, and
$$
0\in \Int(\ch(\sA^{(m)}))
=
\Int(\ch_{[m]}(\sA)).
$$
Thus, Lemma~\ref{lemma:converg-to-contrib-ch}, applied to the reduced families, gives
\[
\ch_{[m]}(\sA^{(n)})
\xrightarrow{H}
\ch\left(
\bigcup_{i\in\sC^{(1)}(\sA^{(m)})} A_i
\right).
\]
Finally, by Lemma~\ref{lemma:ch-is-ch-of-contrib} applied to \(\sA^{(m)}\),
\[
\ch\left(
\bigcup_{i\in\sC^{(1)}(\sA^{(m)})} A_i
\right)
=
\ch(\sA^{(m)})
=
\ch_{[m]}(\sA).
\]
Hence
\[
\ch_{[m]}(\sA^{(n)})\xrightarrow{H}\ch_{[m]}(\sA),
\qquad n\to\infty.
\]
The proof of Theorem~\ref{thm:eventual-equality-layers} is complete.

\subsection{Proof of Theorem~\ref{thm:m-point-peeling-convergence-determ}}
Recall that 
$$
\ch^{[m]}(\sA) = \bigcap_{J\subset I,|J|=m} \cch\left(\bigcup_{i \in I \setminus J} A_i\right)\quad\text{and}\quad\ch^{[m]}(\sA^{(n)}) = \bigcap_{J\subset I_n,|J|=m} \cch\left(\bigcup_{i \in I_n \setminus J} A_i^{(n)}\right).
$$
First, observe that the closures can be removed since the inner unions can be replaced by finite unions only over contributing sets, as we have repeatedly seen before. 

Let us show now, that the intersection in the first (respectively, second) equality above can be restricted exclusively to subsets drawn from the first $m$ contributing layers of the family $\sA$ (respectively, $\sA^{(n)}$). This will allow us to rewrite the right-hand side of the limit statement we are aiming to prove in a simpler way.

By Lemma~\ref{lem:deterministic-layer-reduction} for every $J\subseteq I$, $|J|=m$, we have
\begin{equation*}
    \ch\left(\bigcup_{i \in I \setminus J} A_i\right) = \ch \left( \bigcup_{i \in ( \bigcup_{i = 1}^{m+1}\sC^{(i)}(\sA)) \setminus J} A_i \right).
\end{equation*}
Thus, after deleting $m$ sets we observe a set, which is guaranteed to be defined by sets with indices $i \in \bigcup_{i = 1}^{m+1} \sC_{i}(\sA)$. Note that it will be defined by exactly sets with indices $i \in \bigcup_{i = 1}^{m+1} \sC^{(i)}(\sA)$ only if $J$ contains one index from each $\sC^{(1)}(\sA), \dots, \sC^{(m)}(\sA)$. Consequently, if an index set $J$ of size $m$ contains an index $j \notin \bigcup_{k=1}^{m} \sC^{(k)}(\sA)$, the set $A_j$ is strictly in the interior of the convex hull formed by the other sets. Removing it does not alter the convex hull. This gives us the equality:
\begin{equation*}
    \ch\left(\bigcup_{i \in I \setminus J} A_i\right) = \ch\left(\bigcup_{i \in  I \setminus (J \setminus \{j\})} A_i \right).
\end{equation*}

Now, we can swap $j$ for some contributing set of first $m$ layers of $\sL$ to obtain a non-larger convex hull.
We choose an index $k \in (\bigcup_{k=1}^{m} \sC^{(k)}(\sA)) \setminus J$ and define a new index set $J' = (J \setminus \{j\}) \cup \{k\}$. 

A very important note is the following inclusion, which follows from Corollary~\ref{cor:multi-layer-inclusion} in the Appendix:
\begin{equation*}
    \sC(\sA \setminus \{ A_i: i \in J \setminus \{ j \}\}) \subseteq \bigcup_{i = 1}^m \sC^{(i)}(\sA).
\end{equation*}
Thus, it prevents us from stating that removing set with index $k$, instead of index $j$ will make convex hull strictly smaller. However, it is sufficient for our argument to simply state that the resulting convex hull will be non-larger. 

Because $J'$ removes the set $A_k$, which may make the convex hull smaller, instead of $A_j$, which does not influence the convex hull, the resulting convex hull is bounded by the hull produced by $J$:
\begin{equation*}
    \ch\left(\bigcup_{i \in I \setminus J'} A_i \right) \subseteq \ch\left(\bigcup_{i \in I \setminus (J\setminus \{j\})} A_i\right) = \ch\left(\bigcup_{i \in I \setminus J} A_i\right).
\end{equation*}

By repeating this swapping process for every index in $J$ that does not belong to $\bigcup_{k=1}^{m } \sC^{(k)}(\sA)$, we can construct an index set $\tilde{J} \subseteq \bigcup_{k=1}^{m} \sC^{(k)}(\sA)$ of size $m$ such that its resulting convex hull is a subset of the one produced by $J$. 

Because these supersets are redundant in an intersection, we may restrict the intersection strictly to index sets drawn from the first $m$ layers:
\begin{equation}\label{eq:thm42-proof1}
    \bigcap_{J\subset I,|J|=m} \cch\left(\bigcup_{i \in I \setminus J} A_i \right) = \bigcap_{J \subseteq \bigcup_{k=1}^{m} \sC^{(k)}(\sA),|J|=m} \cch\left(\bigcup_{i \in I \setminus J} A_i\right).
\end{equation}
We shall continue our proof by showing that it is also possible to rewrite the intersection on the left side of the limit statement in a similar way. Let $\sA^{(n)}$ denote the set of atoms of $\mu_n$. Since we assumed that all sets from $\sA^{(n)}$ are strictly convex sets, by Lemma 6.4 from~\cite{mm22} the sets from $\sA^{(n)}$ are in a general position, for sufficiently large $n \in \NN$. Here one has to observe that the general position part of the cited lemma utilizes only conditions (i) and (iii) of that lemma, which do hold in our situation.  Moreover, it is in general position cofinitely because removing a finite collection of sets preserves vague convergence by Lemma~\ref{lemma:vague-conv-removed-atoms} and the limiting reduced process still satisfies the assumption (i) and (iii) of Lemma 6.4 from~\cite{mm22}.  Assume that these sets are in general position starting from some index $N$ and put $n_0 = \max(N, m)$. Applying the same arguments for $\sA^{(n)}, n \geq n_0$ as we have used above for $\sA$, we can write the following equality:
        \begin{equation*}
             \bigcap_{J\subset I_n,|J|=m} \ch\left(\bigcup_{i \in I_n \setminus J} A_i^{(n)}\right)  =  \bigcap_{J \subseteq \bigcup_{k = 1}^{m} \sC^{(k)}(\sA^{(n)}),|J|=m} \ch\left(\bigcup_{i \in I_n \setminus J} A_i^{(n)}\right),\quad n\geq n_0.
        \end{equation*}
Combining this with~\eqref{eq:thm42-proof1} we see that it suffices to prove
        \begin{equation}\label{final-conv}
            \bigcap_{J \subseteq \bigcup_{k = 1}^{m} \sC^{(k)}(\sA^{(n)}),|J|=m} \ch\left(\bigcup_{i \in I_n \setminus J} A_i^{(n)}\right)~\overset{H}{\to}~\bigcap_{J \subseteq \bigcup_{k = 1}^{m} \sC^{(k)}(\sA),|J|=m} \ch\left(\bigcup_{i \in I \setminus J} A_i\right).
        \end{equation}
        By Theorem~\ref{thm:eventual-equality-layers} for all sufficiently large $n$,
        $$
        \bigcup_{k = 1}^{m} \sC^{(k)}(\sA^{(n)})=\bigcup_{k = 1}^{m} \sC^{(k)}(\sA).
        $$
        This holds because Assumption (3) implies, in particular, that $0\in \Int(\ch_{[m+1]}(\sA))$. Thus, intersections on both sides of~\eqref{final-conv} are taken over the same index set for all sufficiently large $n$. By Lemma~\ref{thm:converg-remove-m-points}, for every fixed $J\subseteq \bigcup_{k = 1}^{m} \sC^{(k)}(\sA)$ such that $|J|=m$, we have
        $$
        \ch\left(\bigcup_{i \in I_n \setminus J} A_i^{(n)}\right)~\overset{H}{\to}~\ch\left(\bigcup_{i \in I \setminus J} A_i\right),\quad n\to\infty.
        $$
        It remains to apply Theorem 1.8.10 from~\cite{schneider}, which provides sufficient conditions ensuring continuity of intersection with respect to Hausdorff distance. In our case these conditions holds because sets in focus all contain the origin in the interiors. The proof of Theorem~\ref{thm:m-point-peeling-convergence-determ} is complete.

\subsection{Proof of Theorems~\ref{thm:recursive-peeling-convergence-random} and~\ref{thm:m-point-peeling-convergence-random}}
Put 
$$
\sL_{\Xi_n}:=\{n^{-1}(K-\xi_i)^{o}:1\leq i\leq n\},\quad \sL_{\Pi}:=\{[0,x]:x\in \Pi_K\}.
$$
By Theorem 5.6 in~\cite{mm22}, the following convergence in distribution holds true
\begin{equation*}
    \mu_n:=\sum_{i=1}^n \delta_{\,n^{-1}(K-\xi_i)^o} \overset{d}{\to} \sum_{x\in \Pi_K}\delta_{[0,x]} =: \mu.
\end{equation*}
with respect to the vague topology on the space of point measures on $\sK^d_0\setminus\{0\}$.
By the Skorokhod representation theorem (Theorem 4.30 in \cite{kallenberg-foundations}), we may pass to a new
probability space on which there exist versions of $\mu_n$ and $\mu$ (which we denote by the
same symbols) such that
\begin{equation*}
    \mu_n~\overset{{\rm v}}{\to}~\mu
    \qquad \text{almost surely.}
\end{equation*}
Thus, the key convergence~\eqref{thm:m-point-peeling-convergence-determ-a1} holds true. The only assumption which concerns the pre-limit processes is Assumption (1) in Theorem~\ref{thm:m-point-peeling-convergence-determ}. It holds true because strict convexity and regularity are preserved by the polarity. Thus, if $K$ is strictly convex and regular, then each set $n^{-1}(K-\xi_i)^{o}$ is strictly convex and regular. 

It remains to check that all the assumptions on the limiting point process are satisfied by $\sum_{x\in \Pi_K}\delta_{[0,x]}$. These follow from the following easy facts:

\vspace{2mm}
\noindent

\vspace{2mm}
\noindent
1) With probability one, for every finite $F\subset\Pi_K$, the convex hull
$$
P_F:=\ch\Bigl(\bigcup_{x\in\Pi_K\setminus F}[0,x]\Bigr)
$$
is a polytope containing the origin in its interior. In particular, $\ch(\sL_\Pi)$ is a polytope containing the origin in its interior. Moreover, $D(\sL_\Pi)=+\infty$, and, for every $m\in\mathbb N$, the recursive peeling \(\ch_{[m]}(\sL_\Pi)\) is also a polytope containing the origin in its interior. To see this, note that deleting finitely many points from $\Pi_K$ still leaves infinitely many points in every neighbourhood of the origin and in
every non-empty open cone of directions, while only finitely many points remain outside each ball $B_r(0)$. Hence the remaining points still positively span $\mathbb R^d$, so the corresponding convex hull contains the origin in its interior. Choosing a finite subset of the remaining points whose convex hull already contains the origin in its interior, and then choosing $\rho>0$ such that $B_\rho(0)$ is contained in this finite convex hull, all points of $\Pi_K\setminus F$ lying inside $B_\rho(0)$ are redundant, while only finitely many points lie outside $B_\rho(0)$. Thus, $P_F$ is a polytope. Finally, for such a cofinite family the set of contributing segments is finite and non-empty: finiteness follows because every contributing segment must have its endpoint outside a sufficiently small ball contained in the interior of $P_F$. Therefore, each recursive peeling step removes a finite non-empty collection of segments, and the same argument applies inductively. Thus, the recursive peeling never terminates.

\vspace{2mm}
\noindent
2) With probability one $\sL_{\Pi}$ is in general position cofinitely. This follows from the fact that the collection of points $\{x:x\in \Pi_K\}$ is in general position in the usual sense (no $m+2$ points lie in an affine subspace of dimension $m$, for all $1\leq m\leq d$), see also Example 3.8 in~\cite{mm22}. However, it is not true that every subfamily of $\sL_{\Pi}$ is in general position, the counterexample being $\{[0,x],[0,y]\}\subset \sL_{\Pi}$, $x,y\in \Pi_K$. It is not in general position because $0$ is a $0$-dimensional exposed face of $\ch\{0,x,y\}$ which intersects two members of $\{[0,x],[0,y]\}$.

By the above discussion, $0\in \Int(\ch_{[k]}(\sL))$ for any $k\in\mathbb{N}$ and also $0\in \Int(\ch(\Pi_K\setminus F))$ for any finite set $F\subset \Pi_K$. Fix an arbitrary $m\in\NN$. Applying Theorem~\ref{thm:eventual-equality-layers} separately for each $r=1,\dots,m$, we obtain
\begin{equation*}
    n^{-1}(X_n^o)_{[r]} \overset{H}{\to} (Z^o)_{[r]}
    \qquad \text{almost surely,}
\end{equation*}
for every $r=1,\dots,m$. Consequently,
\begin{equation*}
    \left(n^{-1}(X_n^o)^{[1]},\dots,n^{-1}(X_n^o)^{[m]}\right)
    \longrightarrow
    \left((Z^o)^{[1]},\dots,(Z^o)^{[m]}\right)
    \qquad \text{almost surely}
\end{equation*}
in $(\sK^d_{(0)})^m$, and therefore also in distribution.

This proves Theorem~\ref{thm:recursive-peeling-convergence-random} for peelings. The claim for wrappings follows by the continuous mapping theorem applied to the continuous map $\sK^d_{(0)}\ni L\mapsto L^{o}\in \sK^d_{(0)}$ in view of equality~\eqref{eq:wrapping_duality}. The proof of Theorem~\ref{thm:recursive-peeling-convergence-random} is complete.

The proof of Theorem~\ref{thm:m-point-peeling-convergence-random} follows from Theorem~\ref{thm:m-point-peeling-convergence-determ} via the same scheme.

\section*{Appendix}
\subsection{Proof of Lemma~\ref{lem:deterministic-layer-reduction}}
For the proof of Lemma~\ref{lem:deterministic-layer-reduction} we need yet another auxiliary result, Corollary~\ref{cor:multi-layer-inclusion} below, which itself follows from the next 

\begin{lemma}\label{lem:single-layer-inclusion}
Let $\sL=\{L_i : i\in I\}$ be a family of subsets of $\sK^d$. Let $i_0\in I$ and $1\leq r \leq D(\sL)-2$. Assume that either every subfamily of $\sL$ is in general position or $\sL$ is in general position cofinitely and $|\sC^{(s)}(\sL)|<+\infty$, for all $1\leq s\leq r+1$. Then
\begin{equation*}
  \bigcup_{s=1}^{r}\sC^{(s)}\bigl(\sL\setminus\{L_{i_0}\}\bigr)
  \ \subseteq\
  \bigcup_{s=1}^{r+1}\sC^{(s)}(\sL).
\end{equation*}
\end{lemma}
\begin{proof}
For $1\leq s\leq r$, denote $\sL^{(s)}=\{L_i : i\in I\setminus\bigcup_{q=1}^{s-1}\sC^{(q)}(\sL)\}$. We argue by induction on $r$.

\medskip
\noindent\emph{Base case $r=1$.} If $i_0\in\sC^{(1)}(\sL)$, Lemma~\ref{lemma:remove-1-point-from-ch} (with
$i=i_0$) gives
\begin{equation*}
  \ch\bigl(\sL\setminus\{L_{i_0}\}\bigr)
  =\ch\Bigl(\bigcup_{j\in I\setminus\{i_0\}}L_j\Bigr)
  =\ch\Bigl(\bigcup_{j\in(\sC^{(1)}(\sL)\cup\sC^{(2)}(\sL))\setminus\{i_0\}}L_j\Bigr).
\end{equation*}
Therefore, Lemma~\ref{lem:increase_family} yields
$\sC^{(1)}(\sL\setminus\{L_{i_0}\})\subseteq
\sC^{(1)}(\sL)\cup\sC^{(2)}(\sL)$.
If instead $i_0\notin\sC^{(1)}(\sL)$, 
$\sC^{(1)}(\sL\setminus\{L_{i_0}\})=\sC^{(1)}(\sL)$.
In both cases
$\sC^{(1)}(\sL\setminus\{L_{i_0}\})\subseteq
\sC^{(1)}(\sL)\cup\sC^{(2)}(\sL)$.

\medskip
\noindent\emph{Inductive step.}
Assume $\bigcup_{s=1}^{r-1}\sC^{(s)}(\sL\setminus\{L_{i_0}\})\subseteq
\bigcup_{s=1}^{r}\sC^{(s)}(\sL)$. It suffices to show
\begin{equation}\label{eq:lem-4-9-p1roof1}
\sC^{(r)}(\sL\setminus\{L_{i_0}\})\subseteq\bigcup_{s=1}^{r+1}\sC^{(s)}(\sL). 
\end{equation}
We distinguish two cases.

\smallskip
\noindent\textit{Case 1: $i_0\notin\bigcup_{s=1}^{r}\sC^{(s)}(\sL)$.}
We claim that $\sC^{(s)}(\sL\setminus\{L_{i_0}\})=\sC^{(s)}(\sL)$ for
$s=1,\dots,r$, which is clearly sufficient for~\eqref{eq:lem-4-9-p1roof1}, and prove this by induction on $s$. For $s=1$, we know that $i_0\notin\sC^{(1)}(\sL)$, hence $\sC^{(1)}(\sL \setminus \{ L_{i_0} \}) = \mathcal{C}^{(1)}(\sL)$. This verifies the base of induction. If the equalities hold through some $s<r$, then the first $s$
layers removed from the two families agree.  Moreover,
$i_0\notin\sC^{(s+1)}(\sL)=\sC^{(1)}(\sL^{(s+1)})$, so
we have
$\sC^{(s+1)}(\sL\setminus\{L_{i_0}\})=\sC^{(s+1)}(\sL)$. This completes the inductive step and prove~\eqref{eq:lem-4-9-p1roof1} in Case 1.

\smallskip
\noindent\textit{Case 2: $i_0\in\bigcup_{s=1}^{r}\sC^{(s)}(\sL)$.}
Let $j\in\sC^{(r)}(\sL\setminus\{L_{i_0}\})$ and suppose, for contradiction, that
$j\notin\bigcup_{s=1}^{r+1}\sC^{(s)}(\sL)$ (this includes the case where $L_j$
never contributes to any layer of $\sL$). Then $L_j$ belongs to
$\sL^{(r+1)}$ (because $j\notin\bigcup_{s=1}^{r}\sC^{(s)}(\sL)$) and is non-contributing there (because $j\notin\sC^{(r+1)}(\sL))$). Hence, by Lemma~\ref{lem:noncontrib-interior},
\begin{equation}\label{eq:interior}
  L_j\subset {\rm relint}\ch\bigl(\sL^{(r+1)}\bigr)={\rm relint}\ch_{[r+1]}(\sL).
\end{equation}
We claim $\sL^{(r+1)}\subseteq(\sL\setminus\{L_{i_0}\})^{(r)}$ as families.
The index set of $\sL^{(r+1)}$ is
$I\setminus\bigcup_{s=1}^{r}\sC^{(s)}(\sL)$, while that of
$(\sL\setminus\{L_{i_0}\})^{(r)}$ is
$(I\setminus\{i_0\})\setminus\bigcup_{s=1}^{r-1}\sC^{(s)}(\sL\setminus\{L_{i_0}\})$.
Let $k\in I\setminus\bigcup_{s=1}^{r}\sC^{(s)}(\sL)$. Since
$i_0\in\bigcup_{s=1}^{r}\sC^{(s)}(\sL)$, we have $k\ne i_0$; and by the inductive
hypothesis
$\bigcup_{s=1}^{r-1}\sC_s(\sL\setminus\{L_{i_0}\})\subseteq
\bigcup_{s=1}^{r}\sC_s(\sL)$, which does not contain $k$. Hence $k$ lies in
the second index set, proving the claim. Consequently,
\begin{equation*}
  \ch_{[r+1]}(\sL)=\ch\bigl(\sL^{(r+1)}\bigr)
  \subseteq\ch\bigl((\sL\setminus\{L_{i_0}\})^{(r)}\bigr).
\end{equation*}
Since all the sets are of the same dimension because $r+1 \leq D(\sL)-1$, by~\eqref{eq:interior},
$L_j\subset{\rm relint}\ch\bigl((\sL\setminus\{L_{i_0}\})^{(r)}\bigr)$. By Lemma~\ref{lemma:cont-set-intersect-boundary} $L_j$ is  non-contributing in $(\sL\setminus\{L_{i_0}\})^{(r)}$, that is,
$j\notin\sC^{(1)}\bigl((\sL\setminus\{L_{i_0}\})^{(r)}\bigr)=
\sC^{(r)}(\sL\setminus\{L_{i_0}\})$, a contradiction. Therefore,~\eqref{eq:lem-4-9-p1roof1} holds also in Case 2.
The proof is complete.
\end{proof}

\begin{corollary}\label{cor:multi-layer-inclusion}
Under the assumptions of Lemma~\ref{lem:single-layer-inclusion}, for every fixed finite $J\subseteq I$ the following holds \emph{for all} $1\leq r\leq D(\sL)-|J|-1$ simultaneously:
\begin{equation*}
  \bigcup_{s=1}^{r}\sC^{(s)}\bigl(\sL\setminus\{L_i : i\in J\}\bigr)
  \ \subseteq\
  \bigcup_{s=1}^{\,r+|J|}\sC^{(s)}(\sL).
\end{equation*}
\end{corollary}

\begin{proof}
The induction is on $q:=|J|$ alone. The statement attached to $q$ is the proposition
\begin{multline*}
  P(q):\qquad
  \text{for every $J\subseteq I$ with $|J|=q$ \emph{and every} $1\leq r\leq D(\sL)-q-1$,}\\
  \bigcup_{s=1}^{r}\sC^{(s)}\bigl(\sL\setminus\{L_i : i\in J\}\bigr)
  \subseteq\bigcup_{s=1}^{\,r+q}\sC^{(s)}(\sL).
\end{multline*}

\smallskip
\noindent\emph{Base.} $P(1)$ is simply Lemma~\ref{lem:single-layer-inclusion}.
 
\smallskip
\noindent\emph{Step.} Assume $P(q-1)$ with $q\ge 2$. Fix an arbitrary $J$ with $|J|=q$ and
an arbitrary $1\leq r\leq D(\sL)-q-1$. Write $J=J'\cup\{i_0\}$ with
$i_0\notin J'$, so $|J'|=q-1$, and set $\mathcal{G}:=\sL\setminus\{L_i : i\in J'\}$, which is in general position.
Since $r\leq D(\sL)-q-1$, the reduced family $\mathcal{G}$ has depth at least $r+2$. Indeed, deleting $q-1$ sets cannot destroy all of the first $r+q$ non-terminal layers of $\sL$. Hence Lemma~\ref{lem:single-layer-inclusion} applied to $\mathcal{G}$ at the level $r$, deleting the
single set $L_{i_0}$. We have:
\begin{equation}\label{eq:stepA}
  \bigcup_{s=1}^{r}\sC^{(s)}\bigl(\mathcal{G}\setminus\{L_{i_0}\}\bigr)
  \ \subseteq\
  \bigcup_{s=1}^{r+1}\sC^{(s)}(\mathcal{G}).
\end{equation}
The hypothesis $P(q-1)$ holds for the deletion set $J'$ of cardinality $q-1$ and all $1\leq r\leq D(\sL)-q-1$. Thus, for the right-hand side of~\eqref{eq:stepA} we have
\begin{equation}\label{eq:stepB}
  \bigcup_{s=1}^{r+1}\sC^{(s)}(\mathcal{G})
  \ \subseteq\
  \bigcup_{s=1}^{(r+1)+|J'|}\sC^{(s)}(\sL)
  =\bigcup_{s=1}^{\,r+q}\sC^{(s)}(\sL).
\end{equation}
Combining these inclusions yields
$$
\bigcup_{s=1}^{r}\sC^{(s)}\bigl(\sL\setminus\{L_i : i\in J\}\bigr)=\bigcup_{s=1}^{r}\sC^{(s)}\bigl(\mathcal{G}\setminus\{L_{i_0}\}\bigr)\subseteq \bigcup_{s=1}^{\,r+q}\sC^{(s)}(\sL).
$$
As $J$ and $r$ were arbitrary, $P(q)$ holds, and the induction on $q$ is complete.
\end{proof}

We are now in position to prove Lemma~\ref{lem:deterministic-layer-reduction}.

\begin{proof}[Proof of Lemma~\ref{lem:deterministic-layer-reduction}]
Observe first that $D(\sL)\geq m+2$ implies ${\rm dim}\,\ch\left(\cup_{i\in I\setminus J}L_i\right)=d_0$ for any set $|J|$ of cardinality $m$ (or smaller).

Write $\sL\setminus\{L_i : i\in J_m\}$ for the reduced family, which is in general
position as a cofinite subfamily of $\sL$. By Lemma~\ref{lemma:ch-is-ch-of-contrib} applied to it,
\begin{equation}\label{eq:l44-km}
  \ch\left(\bigcup_{j\in I\setminus J_m}L_j\right)
  =\ch\left(\bigcup_{j\in\sC^{(1)}(\sL\setminus\{L_i : i\in J_m\})}L_j\right),
\end{equation}
Corollary~\ref{cor:multi-layer-inclusion} with $r=1$ and $J=J_m$ gives
\begin{equation*}
  \sC^{(1)}\bigl(\sL\setminus\{L_i : i\in J_m\}\bigr)
  \subseteq\bigcup_{r=1}^{m+1}\sC^{(r)}(\sL).
\end{equation*}
Since the left-hand side does not contain any elements of $J_m$
\begin{equation*}
  \sC^{(1)}\bigl(\sL\setminus\{L_i : i\in J_m\}\bigr)
  \subseteq\Bigl(\bigcup_{r=1}^{m+1}\sC^{(r)}(\sL)\Bigr)\setminus J_m
  \subseteq I\setminus J_m.
\end{equation*}
Combining this with~\eqref{eq:l44-km},
\begin{equation*}
  \ch\Bigl(\bigcup_{j\in I\setminus J_m}L_j\Bigr)
  =\ch\Bigl(\bigcup_{j\in\sC^{(1)}(\sL\setminus\{L_i : i\in J_m\})}L_j\Bigr)
  \subseteq
  \ch\Bigl(\bigcup_{j\in(\bigcup_{r=1}^{m+1}\sC^{(r)}(\sL))\setminus J_m}L_j\Bigr)
  \subseteq
  \ch\Bigl(\bigcup_{j\in I\setminus J_m}L_j\Bigr).
\end{equation*}
All three sets coincide, which is the assertion.
\end{proof}  

\bibliographystyle{apalike}
\bibliography{references_file}

\end{document}